\documentclass[11pt]{amsart}
\usepackage{color}

\newcommand{\Zz}{\mathbb{Z}}
\newcommand{\Cc}{\mathbb{C}}
\newcommand{\Pp}{\mathbb{P}}

\newcommand{\Qq}{\mathbb{Q}}

\newcommand{\Hom}{\operatorname{Hom}}

\newcommand{\la}{\langle}
\newcommand{\ra}{\rangle}

\newcommand{\ii}{{\mathrm i}}

\newcommand{\Oo}{\mathcal{O}}

\newcommand{\Tt}{\mathcal{T}}

\def\Aut{\mathop{\mathrm{Aut}}\nolimits}

\def\dim{\mathop{\mathrm{dim}}\nolimits}

\def\PU{\mathop{\mathrm{PU}}}

\def\Hom{\mathop{\mathrm{Hom}}\nolimits}

\def\mod{\mathop{\mathrm{mod}}\nolimits}

\def\Pic{\mathop{\mathrm{Pic}}}

\def\Tor{\mathop{\mathrm{Tor}}\nolimits}

\makeatletter
\newtheorem*{rep@theorem}{\rep@title}
\newcommand{\newreptheorem}[2]{%
\newenvironment{rep#1}[1]{%
 \def\rep@title{#2 \ref{##1}}%
 \begin{rep@theorem}}%
 {\end{rep@theorem}}}
\makeatother

\newreptheorem{Thm}{Theorem}

\newreptheorem{Cor}{Corollary}

\newreptheorem{Con}{Conjecture}

\newtheorem{thm-int}{Theorem}

\theoremstyle{definition}
\newtheorem{Def-s}[Thm]{Definition}

\def\fake{{\rm fake}}

\newtheorem{theorem}{Theorem}[section]
\newtheorem{lemma}[theorem]{Lemma}
\newtheorem{proposition}[theorem]{Proposition}

\newtheorem{remark}[theorem]{Remark}

\numberwithin{equation}{section}

\begin{document}

\title{Explicit equations of a fake projective plane}

\begin{abstract}
Fake projective planes are smooth complex surfaces of general type with Betti numbers equal to those of the usual projective plane.
They come in complex conjugate pairs and have been classified as quotients of the two-dimensional ball by explicitly written arithmetic
subgroups. In this paper we find equations of a projective model of a conjugate pair of fake projective planes by studying the geometry of
the quotient of such surface by an order seven automorphism.
\end{abstract}

%\date{\today}
\date{April, 2019}
\author{Lev A. Borisov}
\address{Department of Mathematics\\
Rutgers University\\
Piscataway, NJ 08854} \email{borisov@math.rutgers.edu}
\author{JongHae Keum}
\address{School of Mathematics\\
Korea Institute for Advanced Study\\ Seoul 02455 Korea}
\email{jhkeum@kias.re.kr}
\thanks{Borisov was partially supported by the NSF Grants DMS-1201466 and DMS-1601907. Keum was supported by the National Research Foundation of Korea (NRF 2019R1A2C3010487).}

\maketitle

\section{Introduction}
A compact complex surface with the same Betti numbers as the usual
complex projective plane is  called {\it a fake projective plane} if it is not
isomorphic to  the complex projective plane.
A fake projective plane has ample canonical divisor, so it is a smooth (and geometrically connected proper) surface of general type
with geometric genus $p_g=0$ and self-intersection of canonical class $K^2=9$ (this definition extends to arbitrary characteristic).
The existence of a fake projective plane was first proved by Mumford \cite{Mum}. His method was based on the theory of $2$-adic uniformization, and led Ishida and Kato  \cite{IK} to prove the existence of two more  in the $2$-adic approach, recently Allcock and Kato \cite{AK} another from a lattice with torsion.
The second author \cite{K06}
gave a construction of a fake projective plane as a Galois cover of { a singular model of Ishida} elliptic surface which, as described by Ishida \cite{Ish}, is covered (non-Galois) by { Mumford fake projective plane. }

Fake projective planes have Chern numbers $c_1^2=3c_2=9$ and are complex 2-ball quotients by Aubin \cite{Aubin} and Yau \cite{Yau}.  Such ball quotients are strongly rigid by Mostow's rigidity theorem \cite{Mos}, i.e., determined by fundamental group up to
holomorphic or anti-holomorphic isomorphism. Fake projective planes come in complex conjugate pairs by Kharlamov-Kulikov \cite{kk} and
have been classified as quotients of the two-dimensional complex ball by explicitly written co-compact torsion-free arithmetic subgroups of $\PU(2,1)$ by Prasad-Yeung \cite{PY} and Cartwright-Steger \cite{CS}, \cite{CS2}. The arithmeticity of their fundamental groups was proved by Klingler \cite{Kl}. There are exactly 100 fake projective planes total, corresponding to 50 distinct fundamental groups. Cartwright and Steger also computed the automorphism group of each fake projective plane $X$, which is
given by $\Aut(X)\cong N(X)/\pi_1(X),$ where $N(X)$ is the
normalizer of $\pi_1(X)$ in its maximal arithmetic subgroup of $\PU(2,1)$.
 In particular $\Aut(X)\cong \{1\}$, $\Zz_3$, $\Zz_3^2$ or $G_{21}$  where $\Zz_n$ is the cyclic group of order $n$ and $G_{21}$ is the unique non-abelian group of
order $21$. Among the 50 pairs exactly 33 admit non-trivial automorphisms: 3 pairs have $\Aut\cong G_{21}$, 3 pairs have $\Aut\cong \Zz_3^2$ and 27 pairs have $\Aut\cong \Zz_3$.
{  It turns out, for example,  that  Mumford fake plane and Keum fake plane have fundamental groups in the same maximal arithmetic subgroup of  $\PU(2,1)$, but the former has $\Aut\cong \{1\}$ and the latter $\Aut\cong G_{21}$.}

{ On the other hand, in \cite{K08} all possible quotients of fake projective planes were classified, e.g., the $\Zz_7$-quotient of a fake projective plane with $\Aut\cong G_{21}$ is a singular model of an elliptic surface with two multiple fibres and one $I_9$-fibre;  the three pairs of fake projective planes with $\Aut\cong G_{21}$ give three elliptic surfaces, up to complex conjugacy, with induced $\Zz_3$-action:  a $(2,3)$-elliptic surface whose $\Zz_3$-quotient is birational to  Ishida elliptic surface, another $(2,3)$-elliptic surface and a $(2,4)$-elliptic surface. See also \cite{K12} and \cite{K17} for further details.

 In this paper we find equations of a projective model of a conjugate pair of fake projective planes by studying the geometry of
the quotient of such surface by an order seven automorphism. 
The equations are given explicitly by 84 cubics in $\Pp^9$ with coefficients in the field $\mathbb{Q}[\sqrt{-7}]$. 
Their complex conjugate equations define the complex conjugate surface. 
 This pair has the most geometric symmetries among the 50 pairs, in the sense that its automorphism group is $G_{21}$ and its $\Zz_7$-quotient has minimal resolution a $(2,4)$-elliptic surface, which is not simply connected and whose
universal double cover has only one multiple (double) fibre, has the same Hodge numbers as K3 surfaces, but Kodaira dimension 1.
This pair is different from those of Mumford and Keum fake planes, and was disscussed in \cite{K11}.
 
It is an open problem to determine whether the bicanonical map of a given fake projective plane gives an embedding into $\Pp^9$. It has been confirmed affirmatively for several pairs of fake projective planes, including the one in this paper, by the vanishing result of \cite{K13}, \cite{K17}, \cite{CK} and the theorem of Reider \cite{reider} (see also \cite{GKMS}, \cite{dbdc}, where the authors use the term `Keum's fake projective planes' for all fake projective planes with $\Aut\cong G_{21}$). The equations in this paper also provide an explicit proof for the embeddability for the pair.
}

The paper is organized as follows. We describe our main result in Section 2 by presenting the equations of a subscheme in $\Cc\Pp^9$ and indicate the computer calculations that allow one to verify that this subscheme is a fake projective plane. In Section 3 we start the explanation of the process that led us to the equations. Specifically, we discuss the geometry of the minimal resolution of the quotient of a certain fake projective plane by $\Zz_7$ and its universal double cover $X$. In Section 4 we describe the breakthrough calculation that allowed us to identify the image of $X$ under a certain map to $\Cc\Pp^3$ as a specific singular sextic surface. In Sections 5 and 6 we describe additional features of the surface $X$ and explain how we found the field of rational functions of the fake projective plane. In Section 7 we finally explain how we obtained the 84 cubic equations of Section 2. We make a minor comment in Section 8.

%\smallskip
%{\bf Acknowledgements.} Borisov was partially supported by the NSF Grants DMS-1201466 and DMS-1601907. Keum was supported by the National Research Foundation of Korea (NRF 2019R1A2C3010487).

\section{Equations}
%2. Fake projective plane: equations. Argue that this gives a smooth surface.
In this section we write down $84$ explicit degree three equations in ten variables. We argue that they
cut out a fake projective plane with an automorphism group of order 21 with $H_1(Z, \Zz)=\Zz_2^4$. Here $\Zz_m:=\Zz/m\Zz$.

The 84 equations with complex conjugate coefficients cut out another fake projective plane that is complex conjugate to the former.
We identify this pair as the pair of fake projective planes which is $(a=7, p=2, \emptyset, D_32_7)$ in Cartwright-Steger classification \cite{CS2}, or as one of the three pairs in the class $(k=\Qq, \ell=\Qq(\sqrt{-7}), p=2, \Tt_1=\emptyset)$ \cite{CS}, \cite{PY}. { This pair does not belong to the class $(a=7, p=2, \{7\})$ which contains Mumford fake plane $(a=7, p=2, \{7\}, 7_{21})$ and Keum fake plane  $(a=7, p=2, \{7\}, D_32_7)$.
}

\medskip
Let $\Cc \Pp^9$ be a projective space with homogeneous coordinates denoted by $(U_0,U_1,\ldots,U_9)$.
Consider the non-abelian group $G_{21}$ of order $21$ which is a semi-direct product $\Zz_7$ and $\Zz_3$.
We define its action on $\Cc\Pp^9$ by its action on the homogeneous coordinates by
\begin{equation}\label{group_action}
\begin{array}{l}
g_7(U_0:U_1:U_2:U_3:U_4:U_5:U_6:U_7:U_8:U_9) :=
\\
 \hskip 10pt
(U_0: \xi^6 U_1:\xi^5 U_2:\xi^3 U_3:\xi U_4:\xi^2 U_5:\xi^4 U_6: \xi U_7:\xi^2 U_8:\xi^4U_9)\\
g_3(U_0:U_1:U_2:U_3:U_4:U_5:U_6:U_7:U_8:U_9) :=
\\ \hskip 10pt
(U_0:U_2:U_3:U_1:U_5:U_6:U_4:U_8:U_9:U_7) \\
\end{array}
\end{equation}
where $\xi=\exp({\frac {2\pi \ii}7})$ is the primitive seventh root of $1$.

\begin{table}[htp]
\caption{Equations of the fake projective plane 1-24}
\begin{center}
$$
\tiny{
\begin{array}{|cl|}
\hline
eq_1=&
U_1 U_2 U_3 + (1 - \ii \sqrt{7}) (U_3^2 U_4 + U_1^2 U_5 +  U_2^2 U_6) + (10 - 2 \ii \sqrt{7}) U_4 U_5 U_6
  \\[.8em]
     eq_2=&
     (-3 + \ii \sqrt{7}) U_0^3 + (7 + \ii \sqrt{7})(-2U_1 U_2 U_3 + U_7 U_8 U_9-8U_4 U_5 U_6)
 \\[.4em]&
 + 8 U_0( U_1 U_4 +
 U_2 U_5 + U_3 U_6)
  + (6 +
    2 \ii \sqrt{7}) U_0( U_1 U_7 +  U_2 U_8 +  U_3 U_9)

         \\[.8em]
eq_3=&(11 - \ii \sqrt{7}) U_0^3  + 128 U_4 U_5 U_6- (18 + 10 \ii \sqrt{7}) U_7 U_8 U_9  + 64( U_2 U_4^2 +
 U_3 U_5^2    +
 U_1 U_6^2)
  \\[.4em]&
+ (-14 - 6 \ii \sqrt{7}) U_0 (U_1 U_7+U_2 U_8+U_3 U_9) + 8(1 +  \ii \sqrt{7}) (U_1^2 U_8+ U_2^2 U_9+U_3^2 U_7-2 U_1 U_2 U_3 )

    \\[.8em]
 eq_4=&
- (1 + \ii \sqrt{7}) U_0 U_3 (4 U_6 + U_9) +8 (U_1 U_2 U_3 +U_1 U_6 U_9+U_5 U_7 U_9)+16(U_5 U_6  U_7 - U_1^2 U_5 -
  U_3 U_5^2)

\\[.4em]

eq_5=& g_3(eq_4)
\\[.4em]
eq_6=&g_3^2(eq_4)

\\[.8em]
eq_7=&
(12 + 4 \ii \sqrt{7}) U_1 U_2 U_3 + (4 + 4 \ii \sqrt{7})(U_3 U_5 U_8- U_0 U_2 U_5 + 4 U_4 U_5 U_6) + (3 - \ii \sqrt{7}) U_0 U_1 U_7

     \\[.4em]&
   + 8 (U_2 U_4 U_7  +
 U_6 U_7 U_8
      -  U_1^2 U_8 -2U_4 U_6 U_8 )+ (2 + 2 \ii \sqrt{7})( U_3 U_8^2- U_0 U_2 U_8)
 \\[.4em]

eq_8=& g_3(eq_7)
\\[.4em]
eq_9=&g_3^2(eq_7)

  \\[.8em]
eq_{10}=&
 (2 + 6 \ii \sqrt{7} )  U_1 U_2 U_3 + 4(-5 + \ii \sqrt{7}) U_5(U_1^2 +2U_4  U_6) -8
 U_0 (U_2 U_5+U_3U_6) + 8(-1 + \ii \sqrt{7}) U_3 U_5^2

 \\[.4em]&
 +
2 (3 - \ii \sqrt{7} ) U_0 U_1 U_7 -8 U_1^2 U_8 +
(-1 - \ii \sqrt{7}) U_8(U_0 U_2  + 4 U_4 U_9)+ 8(1 + \ii \sqrt{7}) U_3 U_5 U_8 -
 32 U_4 U_6 U_8
 \\[.4em]&
 + 2(1 - \ii \sqrt{7})  (2U_6 U_7 U_8+
 4U_5 U_7 U_9 + 4 U_5 U_6 U_7+  U_7 U_8 U_9  )+
2(3 + \ii \sqrt{7} ) U_3 U_8^2 - 16 U_4 U_5 U_9  +4 U_1 U_9^2

\\[.4em]
eq_{11}=& g_3(eq_{10})
\\[.4em]eq_{12}=&g_3^2(eq_{10})
\\[.8em]

eq_{13}=&-8 \ii \sqrt{7} U_1^2 U_3+ (-7+5 \ii \sqrt{7}) U_0 U_2 U_3+4(-7+\ii \sqrt{7}) U_0 U_6^2+4U_0^2 U_7+(8-8 \ii \sqrt{7}) U_1 U_4 U_7
\\[.4em]&
+4(-5-\ii \sqrt{7}) U_2 U_5 U_7+(8+8 \ii \sqrt{7}) U_3 U_6 U_7+ (-1-5 \ii \sqrt{7}) U_1 U_7^2-8 U_2 U_7 U_8+ (6+6 \ii \sqrt{7}) U_3 U_7 U_9

\\[.8em]

eq_{14}=& 8U_1^2 U_3+2 (3-\ii \sqrt{7}) U_0 U_1 U_5+16 U_3 U_4 U_6-16 U_5^2 U_6+2(1+\ii \sqrt{7}) U_2 U_5 U_7-8U_3 U_6 U_7
\\[.4em]&
+2 (-1-\ii \sqrt{7})  U_3^2 U_8+2 (-1+\ii \sqrt{7}) U_0 U_6 U_9+(-5-\ii \sqrt{7}) U_3 U_7 U_9

\\[.8em]

eq_{15}=&2 (-3-\ii \sqrt{7})  U_1^2 U_3+2(3-\ii \sqrt{7})  U_0 U_2 U_3+4(-1+\ii \sqrt{7}) U_0 U_1 U_5+4(-1-\ii \sqrt{7}) U_3^2 U_5
\\[.4em]&
+8 U_1 U_2 U_6+4(1+\ii \sqrt{7}) U_0 U_6^2-4U_0^2 U_7+(1+\ii \sqrt{7}) U_1 U_7^2+2 (-1+\ii \sqrt{7}) U_0 U_1 U_8+4U_3 U_7 U_9

\\[.8em]

eq_{16}=&(-3+\ii \sqrt{7}) U_2^3+(-3+\ii \sqrt{7}) U_1^2 U_3+4 U_0 U_2 U_3+(-2-2 \ii \sqrt{7}) U_0^2 U_4+8 U_1 U_4^2+8 U_0 U_1 U_5
\\[.4em]&
+(-5-\ii \sqrt{7}) U_1 U_2 U_6+(4+4 \ii \sqrt{7}) U_3 U_4 U_6+2 U_0 U_1 U_8+(3-\ii \sqrt{7}) U_2 U_7 U_8+(2+2 \ii \sqrt{7}) U_3 U_4 U_9

\\[.8em]

eq_{17}=&4(-1-\ii \sqrt{7}) U_2^3+ (5+\ii \sqrt{7}) U_0 U_2 U_3+4(3-\ii \sqrt{7}) U_3^2 U_5+16(1- \ii \sqrt{7}) U_2 U_4 U_5
\\[.4em]&
+4(-1-\ii \sqrt{7}) U_2 U_5 U_7-8 U_1 U_2 U_9+4(1+\ii \sqrt{7}) U_3 U_4 U_9-32 U_5^2 U_9-16 U_5 U_8 U_9

\\[.8em]

eq_{18}=&8U_1^2 U_3+ (-5-\ii \sqrt{7}) U_0 U_2 U_3+4 (1+\ii \sqrt{7}) U_3^2 U_5+4 (1+\ii \sqrt{7})  U_1 U_2 U_6
\\[.4em]&
+16(-1+ \ii \sqrt{7}) U_5^2 U_6+8U_2 U_5 U_7-16 U_3 U_6 U_7+8(-1+\ii \sqrt{7}) U_5 U_6 U_8-8U_3 U_7 U_9

\\[.8em]

eq_{19}=&(-5-\ii \sqrt{7}) U_0^2 U_4-8U_2 U_5 U_7+(-1-\ii \sqrt{7}) U_1 U_7^2+4 U_0 U_1 U_8-4 U_2 U_7 U_8+(-5+\ii \sqrt{7})U_1 U_2 U_9
\\[.4em]&
+2(1-\ii \sqrt{7}) U_3 U_4 U_9+2(1-\ii \sqrt{7}) U_0 U_6 U_9+4 U_3 U_7 U_9+2U_8^2 U_9+2U_0 U_9^2

\\[.8em]

eq_{20}=&4(1+\ii \sqrt{7}) U_1^2 U_3+2 (1-\ii \sqrt{7}) U_0 U_2 U_3-8 U_0^2 U_4+4(-3-\ii \sqrt{7}) U_1 U_4^2-8 \ii \sqrt{7} U_0 U_1 U_5
\\[.4em]&
+8(1- \ii \sqrt{7}) U_2 U_4 U_5+ (5-\ii \sqrt{7}) U_0 U_1 U_8+2 (-5+\ii \sqrt{7}) U_3^2 U_8+16 U_5 U_8 U_9+8 U_8^2 U_9

\\[.8em]

eq_{21}=&(1-\ii \sqrt{7}) U_1^2 U_3-4 U_0 U_1 U_5-8 U_3 U_4 U_6-8 U_0 U_6^2+4 U_1 U_4 U_7+(2-2 \ii \sqrt{7}) U_2 U_5 U_7
\\[.4em]&
+2 U_1 U_7^2-2 U_0 U_1 U_8+(1+\ii \sqrt{7}) U_3^2 U_8+(1-\ii \sqrt{7}) U_2 U_7 U_8+(-1+\ii \sqrt{7}) U_3 U_7 U_9

\\[.8em]

eq_{22}=&-8 U_1^2 U_3+16 U_2 U_4 U_5-8 U_1 U_2 U_6+4(1+\ii \sqrt{7}) U_0 U_6^2+ (1+\ii \sqrt{7}) U_0 U_1 U_8
\\[.4em]&

+8 U_2 U_4 U_8-8 U_5 U_6 U_8+4U_1 U_2 U_9-8 U_3 U_4 U_9+2(1+\ii \sqrt{7}) U_0 U_6 U_9

\\[.8em]

eq_{23}=&(-3+\ii \sqrt{7}) ( U_2^3+ U_1^2 U_3)+4(-1-\ii \sqrt{7}) U_1 U_4^2+(-1+3 \ii \sqrt{7}) U_1 U_2 U_6
+2 (-1-\ii \sqrt{7}) U_1 U_4 U_7
\\[.4em]&
+ (1+\ii \sqrt{7}) U_3^2 U_8+8 U_2 U_4 U_8+4(-1+\ii \sqrt{7}) U_5 U_6 U_8+4U_2 U_7 U_8+4U_1 U_2 U_9

\\[.8em]

eq_{24}=&2U_0 U_2 U_3+ (-1-\ii \sqrt{7}) U_0^2 U_4+2(1-\ii \sqrt{7}) U_0 U_1 U_5+2(1-\ii \sqrt{7}) U_1 U_2 U_6+2U_0 U_1 U_8
\\[.4em]&
-4 U_3^2 U_8-4 U_2 U_4 U_8+2(1-\ii \sqrt{7}) U_5 U_6 U_8+4 U_5 U_8 U_9+2U_8^2 U_9

\\
 \hline
\end{array}
}
$$

\end{center}
\label{eqs1-24}
\end{table}

\medskip

\begin{table}[htp]
\caption{Equations of the fake projective plane 25-84}
\begin{center}
$$
\tiny{
\begin{array}{|cl|}
\hline
eq_{25}=&(-1+3 \ii \sqrt{7}) U_0^2 U_1+(44-4 \ii \sqrt{7}) U_2^2 U_3+64 U_3 U_4 U_5+(36-12 \ii \sqrt{7}) U_1 U_3 U_6+(16+16 \ii \sqrt{7}) U_4^2 U_6
\\[.4em]&
+(-4-4 \ii \sqrt{7}) U_0 U_2 U_7-32 U_3 U_4 U_8+(4+4 \ii \sqrt{7}) U_0 U_6 U_8-16 U_3 U_7 U_8+(8-8 \ii \sqrt{7}) U_1 U_3 U_9+16 U_4 U_7 U_9

\\[.8em]

eq_{26}=&(-1+3 \ii \sqrt{7}) U_0^2 U_1+(-4-4 \ii \sqrt{7}) U_2^2 U_3+(40-8 \ii \sqrt{7}) U_1 U_2 U_5+(4-12 \ii \sqrt{7}) U_1 U_3 U_6+96 U_4^2 U_6
\\[.4em]&
+(-24-8 \ii \sqrt{7}) U_2 U_6^2+16 U_1^2 U_7+(-2+2 \ii \sqrt{7}) U_0 U_2 U_7+64 U_4 U_6 U_7+(20-4 \ii \sqrt{7}) U_1 U_2 U_8-8 U_0 U_6 U_8
\\[.4em]&
+16 U_4 U_7 U_9

\\[.8em]

eq_{27}=&(5+\ii \sqrt{7}) U_0^2 U_1+(-4-4 \ii \sqrt{7}) U_2^2 U_3+(16-16 \ii \sqrt{7}) U_3 U_4 U_5+(-20-4 \ii \sqrt{7}) U_1 U_3 U_6+32 U_4^2 U_6
\\[.4em]&
+32 U_0 U_5 U_6+8 U_0 U_6 U_8-16 U_1 U_3 U_9+16 U_0 U_5 U_9+8 U_0 U_8 U_9

\\[.8em]

eq_{28}=&8 U_2^2 U_3+(-3+\ii \sqrt{7}) U_0 U_3^2+(-4-4 \ii \sqrt{7}) U_1 U_2 U_5+(4+4 \ii \sqrt{7}) U_3 U_4 U_5+32 U_5^3+(4+4 \ii \sqrt{7}) U_3 U_5 U_7
\\[.4em]&
+16 U_5^2 U_8+(3-\ii \sqrt{7}) U_1 U_3 U_9+8 U_2 U_6 U_9

\\[.8em]

eq_{29}=&(-3+\ii \sqrt{7}) U_2^2 U_3+(5+\ii \sqrt{7}) U_0 U_2 U_4+8 U_1 U_2 U_5-8 U_2 U_6^2+2 U_0 U_2 U_7+(-1-\ii \sqrt{7}) U_1 U_2 U_8+8 U_5^2 U_8
\\[.4em]&
+(3-\ii \sqrt{7}) U_1 U_3 U_9+(4+4 \ii \sqrt{7}) U_4^2 U_9-8 U_2 U_6 U_9+(2+2 \ii \sqrt{7}) U_4 U_7 U_9-2 U_0 U_8 U_9+(-3+\ii \sqrt{7}) U_2 U_9^2

\\[.8em]

eq_{30}=&8 U_2^2 U_3+(4-4 \ii \sqrt{7}) U_1^2 U_4+(-12-4 \ii \sqrt{7}) U_1 U_2 U_5+(-4-12 \ii \sqrt{7}) U_4^2 U_6+(12+4 \ii \sqrt{7}) U_2 U_6^2
\\[.4em]&
+(2-2 \ii \sqrt{7}) U_1^2 U_7-8 U_1 U_2 U_8-16 U_3 U_4 U_8+(1+3 \ii \sqrt{7}) U_0 U_6 U_8+(-3-\ii \sqrt{7}) U_3 U_7 U_8+4 U_1 U_3 U_9
\\[.4em]&
+(6+2 \ii \sqrt{7}) U_2 U_6 U_9

\\[.8em]

eq_{31}=&(-4+4 \ii \sqrt{7}) U_1^2 U_4-4 U_1 U_2 U_5+(-4+4 \ii \sqrt{7}) U_3 U_4 U_5+16 U_5^3+(-8+8 \ii \sqrt{7}) U_4^2 U_6+(2+2 \ii \sqrt{7}) U_0 U_5 U_6
\\[.4em]&
-4 U_1^2 U_7+(2+2 \ii \sqrt{7}) U_6 U_7^2+8 U_3 U_4 U_8-4 U_0 U_6 U_8-4 U_5 U_8^2+(1+\ii \sqrt{7}) U_7^2 U_9

\\[.8em]

eq_{32}=&(-5-\ii \sqrt{7}) U_0^2 U_1+(-6+2 \ii \sqrt{7}) U_0 U_3^2+(-24+8 \ii \sqrt{7}) U_3 U_4 U_5+(20+4 \ii \sqrt{7}) U_1 U_3 U_6-32 U_4^2 U_6
\\[.4em]&
-32 U_0 U_5 U_6+32 U_2 U_6^2+(2+2 \ii \sqrt{7}) U_0 U_2 U_7+(4+4 \ii \sqrt{7}) U_1 U_2 U_8-8 U_0 U_6 U_8+(10+2 \ii \sqrt{7}) U_1 U_3 U_9
\\[.4em]&
+16 U_2 U_6 U_9

\\[.8em]

eq_{33}=&(7-5 \ii \sqrt{7}) U_0^2 U_1+(-56-24 \ii \sqrt{7}) U_1^2 U_4+32 \ii \sqrt{7} U_1 U_2 U_5+(28+4 \ii \sqrt{7}) U_1 U_3 U_6+(28+28 \ii \sqrt{7}) U_0 U_5 U_6
\\[.4em]&
+(-84-4 \ii \sqrt{7}) U_1^2 U_7+(7+7 \ii \sqrt{7}) U_0 U_2 U_7-56 U_3 U_5 U_7+56 U_6 U_7^2+24 \ii \sqrt{7} U_1 U_2 U_8+56 U_0 U_6 U_8
\\[.4em]&
+(14-18 \ii \sqrt{7}) U_1 U_3 U_9+28 U_7^2 U_9

\\[.8em]

eq_{34}=&(-5-\ii \sqrt{7}) U_0^2 U_1+48 U_1 U_2 U_5+(-16-16 \ii \sqrt{7}) U_3 U_4 U_5+32 U_4^2 U_6+(2+10 \ii \sqrt{7}) U_1^2 U_7
\\[.4em]&
+(-48+16 \ii \sqrt{7}) U_4 U_6 U_7
+(28-4 \ii \sqrt{7}) U_1 U_2 U_8+(-12-12 \ii \sqrt{7}) U_3 U_4 U_8+(-16-8 \ii \sqrt{7}) U_0 U_6 U_8
\\[.4em]&
+(-22+2 \ii \sqrt{7}) U_1 U_3 U_9
+(-8-8 \ii \sqrt{7}) U_2 U_6 U_9+(-8+8 \ii \sqrt{7}) U_4 U_7 U_9

\\[.8em]

eq_{35}=&(10+2 \ii \sqrt{7}) U_2^2 U_3+(-11+\ii \sqrt{7}) U_0 U_2 U_4-16 U_1 U_2 U_5+(20+4 \ii \sqrt{7}) U_3 U_4 U_5-16 U_2 U_6^2
\\[.4em]&
+(-1-\ii \sqrt{7}) U_0 U_2 U_7
+(-2-2 \ii \sqrt{7}) U_1 U_2 U_8-16 U_5^2 U_8+(-4-4 \ii \sqrt{7}) U_4^2 U_9+(3-\ii \sqrt{7}) U_2 U_9^2

\\[.8em]

eq_{36}=&(2+2 \ii \sqrt{7}) U_0 U_3^2+(-6+2 \ii \sqrt{7}) U_0 U_2 U_4+(4-4 \ii \sqrt{7}) U_1 U_3 U_6+32 U_4^2 U_6+(-12-4 \ii \sqrt{7}) U_2 U_6^2+2 U_0 U_2 U_7
\\[.4em]&
+16 U_4 U_6 U_7+(7-\ii \sqrt{7}) U_1 U_2 U_8-8 U_5^2 U_8+4 U_1 U_3 U_9+(4-4 \ii \sqrt{7}) U_0 U_5 U_9+4 U_4 U_7 U_9-2 \ii \sqrt{7} U_0 U_8 U_9

\\[.8em]

eq_{k}=&g_3(eq_{k-24}),~~k=37,\ldots,60

\\[.8em]
eq_{k}=&g_3^2(eq_{k-48}),~~k=61,\ldots,84

    \\
\hline
\end{array}
}
$$
\end{center}
\label{eqs25-84}
\end{table}

\begin{theorem}\label{main}
Eighty four cubic equations of Tables \ref{eqs1-24} and \ref{eqs25-84}
 give equations of a fake projective plane $Z$ in $\Cc\Pp^9$ embedded by its bicanonical linear system.
\end{theorem}

\begin{proof}
{ Let $Z$ be the subscheme of $\Cc\Pp^9$ cut out by these eighty four equations.}
We used Magma to calculate the Hilbert series of $Z$ to give
$$
\dim H^0(Z,\Oo(k))
% = \frac 12 (6k-1)(6k-2)
=18k^2-9k+1
$$
for all $k\geq 0$.

\medskip
We then used reduction modulo $263$ with $\ii\sqrt 7 = 16\mod 263$ { (which was chosen just because it is a decent size prime with a clear root of $-7$). }  We calculated (by Macaulay2) the projective resolution
of $\Oo_Z$ as
$$
\begin{array}{c}
0\to
\Oo(-9)^{\oplus 28}\to
\Oo(-8)^{\oplus 189}\to
\Oo(-7)^{\oplus 540}\to\Oo(-6)^{\oplus 840}\to
\\
\to\Oo(-5)^{\oplus 756}\to\Oo(-4)^{\oplus 378}\to\Oo(-3)^{\oplus 84}\to\Oo\to\Oo_Z \to 0.
\end{array}
$$
By semicontinuity, the resolution is of the same shape over $\Cc$. Since all of the sheaves $\Oo(-k),  ~k=3,\ldots,9$ are acyclic, we see that { for all $i\ge 0$ $$h^i(Z,\Oo_Z)=h^i(\Cc\Pp^9,\Oo).$$ That is, $h^1(Z,\Oo_Z)=h^2(Z,\Oo_Z)=0$ and $h^0(Z,\Oo_Z)=1$, which implies that the scheme $Z$ is connected. Since the Hilbert polynomial has degree 2, its irreducible components have dimension at most 2.}

\medskip
We also verified that $Z$ is smooth. It is a somewhat delicate calculation. In theory, one can take the $7\times 7$ minors of the $84\times 10$ matrix of partial derivatives of the  equations
and verify that, together with the equations themselves, they generate the ideal which coincides with $\Cc[U_0,\ldots,U_9]$ for large degrees. In practice, such direct calculation is impossible,
since the number of minors is too large.
{ Instead, we pick three $7\times 7$ minors of the Jacobian matrix and show that they have no common zeros on $Z$ by a Hilbert polynomial calculation.
The minors were chosen so that they do not vanish at the fixed points of the automorphism $g_7$, namely at the three points}
%Instead, we took all minors of Jacobian matrices of three subsets of size $7$ of the set of equations which were picked by the condition that the equations are
%linearly independent to first order at the fixed points
$$
(U_0,\ldots,U_7,U_8,U_9)\in \{(0,\ldots,0,0,1),(0,\ldots,0,1,0),(0,\ldots,1,0,0)\}.
$$
{ The subsets of  equations and variables that define the minors are given in Table \ref{3minors}.}
%These subsets of indices are given in Table \ref{which.minors}.
This calculation was performed in Magma software package
modulo { $263$} with $\ii\sqrt{7}=16$.
{ The Hilbert polynomial of the quotient drops from $18k^2-9k+1$ to $504k-3654$, then to $7056$ and finally to $0$ as one adds the three minors to the ideal.}
%The Hilbert polynomial of the quotient drops from $18n^2-9n+1$ to $378n-2079$, then to $3969$ and finally to $0$ as one adds the minors to the ideal.
If the equations generate the ring modulo { $263$,} then they also generate it with exact coefficients. { This calculation means that all geometric points of $Z$ have tangent space of dimension at most two, which together with $h^{0}(\Oo_Z)=1$ implies that $Z$ is a smooth surface.}

\begin{table}[htp]
\caption{Three $7\times 7$ minors used to verify smoothness}
%\caption{Indices of seven equations $eq_i$ whose $7\times 7$ minors of the Jacobian matrix are used to verify smoothness}
$$\begin{array}{|c|}
\hline
 \{ 8,19,29,43,55,61,79 \} { ; \{U_0,U_1,U_2,U_3,U_5,U_6, U_7\}} \\
\{ 7,19,31,37,55,67,77 \} { ; \{U_0,U_1,U_2,U_3,U_4,U_5, U_9\}}\\
\{ 9,13,31,43,53,67,79\} { ; \{U_0,U_1,U_2,U_3,U_4,U_8, U_9\}}\\
\hline
\end{array}
$$
\label{3minors}
\end{table}

\medskip
Thus we have a smooth surface $Z$ and a very ample divisor class $D=\Oo_Z(1)$ on it. The Hilbert polynomial together with the Riemann-Roch implies that
$$D^2=36,\,\,\, D K_Z=18,\,\,\, \chi(Z,\Oo_Z)=1.$$ Note that this shows that $Z$ is not isomorphic to $\Cc\Pp^2$.
{ We also know that $h^{0,1}(Z)=h^{0,2}(Z)=0$, so it remains to prove that $h^{1,1}(Z)=1$.}

\medskip
To figure out this last Hodge number we used Macaulay to calculate $\chi(Z,\Oo(2K_Z))=10$ (again working modulo $263$). { For this calculation, we used the resolution to compute the canonical bundle $K_Z$ as in
Hartshorne's book, as Ext from the canonical bundle of the ambient space to $\Oo_Z$, then
tensored it with self to get $2K_Z$, then calculated the Hilbert polynomial of
the corresponding graded module to get $\chi(Z, 2K_Z)$. Now by Riemann-Roch
$\chi(Z, 2K_Z) = K_Z^2 + \chi(Z, \Oo_Z)$ and we know that $\chi(Z,\Oo_Z)=1$, thus $K_Z^2=9$.
}
Now Noether's formula finishes the proof that $Z$ is a fake projective plane.

\medskip
We see that $2K$ is numerically equivalent to $D$.
We calculated
$$\Hom(\Oo(K),\Oo(D))=0$$
by working modulo $263$ and semi-continuity.
This implies
$$h^0(Z,\Oo(D-K))=0=h^2(Z,\Oo(2K-D)).$$
This implies that
$h^0(Z,\Oo(2K-D))\geq 1$, so $\Oo(2K)\simeq \Oo(D)$. So the fake projective plane $Z$ is embedded via a bicanonical embedding.
\end{proof}

\begin{remark}
Consider the closed subscheme $C$ of $Z$ cut out by $U_0=0$ and the following 18 quadrics, which fall into 6 orbits under the
%$C_3$-action
{ $\Zz_3$-action}
$$
\la U_0,\,\,\, U_1^2 - U_6 U_7 + \frac 18(-5 - \ii\sqrt 7)U_7U_9,
%U_3 U_5 + \frac 18 (1 + \ii\sqrt 7) U_7 U_9,
\,\,\, U_4 U_6 - \frac 18 (1 + \ii\sqrt 7)U_3 U_8,$$
$$U_2 U_4 + \frac 18 (1 + \ii\sqrt 7) U_8 U_9,\,\,\,
%$$(-1 + \ii\sqrt 7) U_5 U_6 +U_2 U_7, U_3^2 - U_5 U_9 + \frac 18 (-5-\ii\sqrt 7) U_8 U_9,$$
U_1U_4 + U_3U_6 + \frac 18 (5 + \ii\sqrt 7)U_3U_9,
$$
$$U_1U_2 + U_5 U_8,\,\,\,
U_4^2 +  \frac 18 (1 + \ii\sqrt 7)U_2U_9 + \frac 18 (5 + \ii\sqrt 7)U_4U_7,$$
$$12\,\,{\rm images\,\, of\,\,the\,\, six\,\,quadrics\,\, under}\,\,
%C_3
{ \Zz_3}
 \ra.
$$

\medskip\noindent
By calculating its Hilbert polynomial, we see that it is one-dimensional, with the total degree of one-dimensional components equal to $18$.
This means that $C$ is a (manifestly $\Zz_7$-invariant) curve on $Z$. Moreover, { by computing Hilbert polynomials of $\la U_0 \ra+ I^2$ and $\la U_0\ra$,} we see that the square of this ideal $I$ lies in $\la U_0\ra$.
Therefore, we see that the bicanonical divisor $2K_Z$ is linearly equivalent to $2C$. By Lemma \ref{24} this implies that $Z/\Zz_7$ has minimal model
which is a $(2,4)$-elliptic surface. It also identifies $Z$ as the pair of fake projective planes which is $(a=7, p=2, \emptyset, D_32_7)$ in Cartwright-Steger classification \cite{CS2}, or as one of the three pairs in the class $(k=\Qq, \ell=\Qq(\sqrt{-7}), p=2, \Tt_1=\emptyset)$ \cite{CS}, \cite{PY}.
\end{remark}

\begin{lemma}\label{24} Let $W$ be a fake projective plane with $\Aut(W)=\Zz_7:\Zz_3$. Then the following are equivalent.
\begin{enumerate}
\item $W$ contains an effective $\Zz_7$-invariant curve $C$ with $C^2=9$.
\item The action of $\Zz_7$ on $W$ fixes a non-trivial element in $H_1(W, \Zz)$.
\item $H_1(W, \Zz)=\Zz_2^4$.
\item { The minimal resolution of $W/\Zz_7$} is a $(2,4)$-elliptic surface.
\end{enumerate}
\end{lemma}

\begin{proof} On a fake projective plane an effective curve $C$ with $C^2=9$ is a member of the linear system $|K_W+t|$ for some non-zero $t\in \Tor\Pic(W)\cong H_1(W, \Zz)$. For a subgroup $G$ of $\Aut(W)$ the linear system $|mK_W+t|$ is $G$-invariant if and only if so is $t$.  For a { cyclic} subgroup $G$ of $\Aut(W)$ a complete linear system is $G$-invariant if and only if a member of the system is $G$-invariant. %(commuting matrices in $\PGL(n)$ are simultaneously triangulable).
This proves the equivalence of (1) and (2). These two are equivalent to (3) by Corollary 3.4 of \cite{CK}, then to (4) by the classification of \cite{K08} on the possible geometric structures of quotients of fake projective planes.
\end{proof}

Furthermore, if $H_1(W, \Zz)=\Zz_2^4$, { a  $\Zz_7$-invariant non-trivial 2-torsion is unique (\cite{CK}, Corollary 3.4), hence is $\Aut(W)$-invariant.}
%If two $\Zz_7$-invariant curves $C_i$ with $C_i^2=9$ exist, then both correspond to the same non-trivial element in $H_1(W, \Zz)$ which is the unique $\Zz_7$-invariant 2-torsion.

{ 
\begin{remark}  It is always true that a surface in $\Pp^9$ can be defined scheme-theoretically
by at most 10 equations. 
Suitably chosen 15 among the 84 equations are enough to define $Z$, as pointed out by a referee.
 However, it is important for constructing the resolution
to cut it out ideal-theoretically. Moreover, Macaulay 2 works smoothly with the saturated ideal generated by the 84 cubics.
\end{remark}
}

\section{Explanation begins: the double cover of the resolution of the $\Zz_7$-quotient of a fake projective plane.}
%3. Explanation begins: double cover of resolution of quotient of fake projective plane.
%   Picard group of the double cover. Divisor 3F+S and its freeness.
The equations of the previous section appear quite mysterious, and we will spend the rest of the paper explaining
their origin. Our general construction can be roughly summarized in the following commutative diagram of morphisms, with notations that will be used throughout the paper
 \medskip

 $\Pp_{\fake}^2$: a fake projective plane with $\Aut=\Zz_7:\Zz_3$ such that
the minimal resolution $Y$ of $\Pp_{\fake}^2/\Zz_7$ is a $(2,4)$-elliptic surface;

 \medskip
 $X$: the universal double cover of $Y$.

\begin{equation}\label{eq_diagram}
\begin{array}{rcccccccc}
{\mathcal B}^2 & && \widehat{\Pp_{\fake}^2} && & \hskip -5pt X &\stackrel \pi{\longrightarrow} & \Pp^3\\
                        &\searrow &&\hskip -40pt  \swarrow            &  \hskip -30pt \searrow     &\swarrow&    &\hskip -20pt     \searrow     &      \\
                        & &   \Pp_{\fake}^2             &                   & \hskip 10ptY&    &                     &               &\hskip -35pt \Pp^1    \\
                         & &  &  \hskip -60pt \searrow        &  \hskip -30pt \swarrow    &\searrow&    &\hskip -20pt      \swarrow &      \\
                        & & &  \Pp_{\fake}^2 /\Zz_7                               &  &              &\Pp^1&&    \\

\end{array}
\end{equation}

\medskip
In this section we describe the known results of \cite{K08}, \cite{K11}, \cite{K17},  on the quotients of fake projective planes
with automorphism group of order $21$ by the subgroup of order $7$. Specifically, we describe the aspects of the geometry of $Y$ and $X$  in \eqref{eq_diagram} that will be later
used to find the equation of $\pi(X)\subset \Pp^3$.

\medskip
Let $\Pp_{\fake}^2$ be a fake projective plane with non-commutative automorphism group $G_{21}\cong \Zz_7:\Zz_3$. Consider the quotient
$\Pp_{\fake}^2/\Zz_7$ of $\Pp_{\fake}^2$ by the (normal) Sylow 7-subgroup of $G_{21}$. It is a singular surface of Kodaira dimension one with three quotient singular points of type $\frac 17(1,3)$ and inherits
an order three automorphism which permutes these singular points.
The minimal resolution $Y$ of $\Pp_{\fake}^2/\Zz_7$ is
an elliptic surface over $\Cc\Pp^1$ with $h^{2,0}(Y)=h^{1,0}(Y)=0$, and two multiple fibers with multiplicities $(2,3)$ or $(2,4)$, as shown  in \cite{K08}.
The Hodge numbers of $Y$ are given by
$$
h^{0,0}(Y) = h^{2,2}(Y)=1,~h^{1,1}(Y)=10,~h^{p,q}(Y)=0 \mbox{ otherwise.}
$$
Throughout the rest of the paper we will consider the fake projective planes which lead to elliptic surfaces $Y$ with multiple fibers
of multiplicities $(2,4)$. By the classification of \cite{PY} and \cite{CS} there is exactly one such conjugate pair
of fake projective planes. (The other two conjugate pairs with $\Aut\cong G_{21}$ lead to $(2,3)$-elliptic surfaces.)
Let us denote by $4F_Y$ the multiplicity four fiber and by $2F_{2,Y}$ the multiplicity two fiber. We summarize the results
of \cite{K08}, \cite{K11} and \cite{K17}.

\medskip
The preimages of $\frac 17(1,3)$ singular points in $Y$ are three  pairwise disjoint chains of spheres
$$
S - B - C,~~~~S' - B' - C', ~~~~ S'' - B'' - C''
$$
with $S^2=(S')^2=(S'')^2=-3$ and the squares of the rest equal to $-2$. The canonical class $K_Y$ is numerically equivalent to $F_Y$, and the elliptic fibration $Y\to \Pp^1$ is given by the linear system
$$|4F_Y|=|2F_{2,Y}|=|4K_Y|,$$ i.e., a general fiber is linearly equivalent to $4F_Y$. The curves $S$, $S'$ and $S''$ are
$4$-sections of the fibration, i.e. $$F_Y S = F_YS'=F_YS''=1.$$ The curves $B$, $C$ and their translates are part of an $I_9$-fiber of $Y\to \Pp^1$
and the order 3 automorphism group acts fiberwise. There are three additional $I_1$-fibers, some of which may be the multiple fibers.

\medskip
The structure of the $I_9$-fiber will be very important in what follows.
We denote its nine components by
$$
A - B - C - A' - B' -C' - A'' - B'' - C'' -A.
$$
The curve $S$ intersects $B$ transversely and does not intersect $C,B',C',B'',C''$. The $I_9$-fiber is not a multiple fiber, i.e., is equivalent to $4F_Y$ by Theorem 2.3 of \cite{K17}, thus we see that $S$ must
intersect $A$, $A'$ and $A''$ in three points total. These intersection numbers determine the intersection numbers of $S'$ and $S''$ with $A,A',A''$
because of the order three automorphism.

\medskip
It is easy to see that the classes of the curves
$$A,B,C,A',B',C',A'',B'',C'',S,S',S''$$
generate a sublattice of rank $\ge 10$ inside the Picard lattice of $Y$, the N\'eron-Severi group of $Y$ modulo torsion. (The first 9 curves already generate a rank 9 sublattice.) By Poincar\'e duality the Picard lattice of $Y$ is unimodular of signature $(1,9)$, thus the sublattice must have rank 10 and discriminant a square integer,
which puts strong restrictions on the intersection numbers. It was observed in \cite{K11} that the only possibilities are
\begin{equation}\label{cases}
\begin{array}{l}
\mbox{Case 1: } S A = 1,~S A' = 0,~SA''=2;\\
\mbox{Case 2: } S A = 0,~S  A' = 2,~S A''=1.
\end{array}
\end{equation}
The fundamental group of a $(2,4)$-elliptic surface is of order 2 \cite{D}, thus the surface $Y$ has an unramified double cover $X$ which is part of the diagram \eqref{eq_diagram}. It comes from a double cover $\Pp^1\to \Pp^1$ of the base of the fibration
ramified over the images of $F_Y$ and $F_{2,Y}$.
The preimage of the canonical divisor class  $K_Y$ is the canonical class $K_X$ of $X$ and is numerically equivalent to the preimage $F$ of $F_Y$. Since on a simply connected surface a numerical equivalence is a linear equivalence, $K_X$ is linearly equivalent to $F$.
We will denote by $F_{2}$ the preimage of $F_{2,Y}$. Then the elliptic fibration $X\to \Pp^1$ is given by the linear system
$$|2F|=|F_{2}|=|2K_X|,$$ and has only one multiple fiber $2F$ (with multiplicity 2). In particular $X$ has Kodaira dimension 1. Since $X$ is simply connected, $h^{1,0}(X)=0$.
Since $\chi(X,\Oo_X)=2\chi(Y,\Oo_Y)=2$ we get $h^0(X,K_X)=1$.  This implies that
$$
h^{0,0}(X) = h^{2,2}(X)=h^{2,0}(X) = h^{0,2}(X)=1,\,\,~h^{1,1}(X)=20,$$
$$
h^{p,q}(X)=0 \mbox{ otherwise,}
$$
i.e., $X$ has the Hodge numbers of K3 surfaces. Its Jacobian fibration is an elliptic surface over $\Pp^1$ with a section, with no multiple fibre, and with singular fibers of the same type as those of $X$ (this is true for Jacobian fibration of any genus one fibration, cf. \cite{CD}), thus has trivial canonical class and the sum of Euler numbers of singular fibres 24, hence is a K3 surface.

The preimage under $X\to Y$ of the curve $S$ is $S_1+S_2$ where $S_i$ are disjoint smooth rational curves with $S_i^2=-2$. Each of the curves $S_i$ is a $2$-section of
$X\to \Cc\Pp^1$. Similarly, we define $S_1'$, $S_2'$, $S_1''$ and $S_2''$. Preimage of the $I_9$-fiber $A-B-\ldots-C''-A$ is two disjoint $I_9$-fibers
$$A_1-B_1-\ldots -C_1''-A_1,\,\,\,\, A_2-B_2-\ldots-C_2''-A_2.$$ We arrange the indexing so that we get six $(-3)-(-2)-(-2)$ chains of $\Cc\Pp^1$ curves
$$
S_i-B_i-C_i,~S_i'-B_i'-C_i',~S_i''-B_i''-C_i'',~~i\in \{1,2\}.
$$

As before, we would like to determine the possible intersection numbers of the $24$ curves
$$S_1,\ldots,S_2'',A_1,\ldots,C_2''$$
with each other. These intersections are uniquely determined by the nonnegative integers
$S_1 A_1, S_1 A_1', S_1 A_1'', S_2 A_1, S_2 A_1',S_2 A_1''$ which are subject to
$S_i(A_1+A_2)=S A$, $S_i(A_1'+A_2')=S A'$, $S_i(A_1''+A_2'')=S A''$ from \eqref{cases}.
The resulting intersection matrix has to have rank at most $20$, because the rank of the Picard group does not exceed $h^{1,1}(X)$.

A simple computer calculation shows that only Case 2 of \eqref{cases} is possible and, moreover, there holds
\begin{equation}\label{intersections_on_X}
S_{1}A_1=S_2 A_1=0,\,\,S_1A_1'=S_2A_1' =1,\,\,S_1 A_1'' =0,\,\,S_2 A_1'' =1,
\end{equation}
i.e., $S_{1}$ intersects at one point exactly $B_1$ and $A_1'$ of the first $I_9$-fiber, and exactly $A_2'$ and $A_2''$ of the second. ($S_{2}$ intersects exactly $B_2$ and $A_2'$ of the second $I_9$-fiber, and exactly $A_1'$ and $A_1''$ of the first.)
This gives a rank $19$ intersection matrix. This rank is not the maximum possible $h^{1,1}(X)=20$, thus leaves a possibility that $F$ or $F_{2}$ is of type $I_2$, i.e., $F_Y$ or $F_{2,Y}$ on $Y$ is of type $I_1$.

The following is crucial in our approach.

\begin{proposition}\label{crucial}
Let  $D$ be the divisor $3F+S_1+S_2$ on $X$ which is the pullback of the divisor $3F_Y+S$ from $Y$.
Then  $D^2=6$, $h^0(D)=4$ and the linear system $|D|$ is base point free. It gives a birational map $\pi: X\to \Cc\Pp^3$ whose image is a sextic surface. Moreover,
\begin{enumerate}
\item $F$ is an elliptic curve and maps $2:1$ onto a line;
\item each $I_9$-fiber maps to a union of  a conic and two lines;
 \item a general fiber maps birationally onto a plane quartic curve with nodes at the points $\pi(S_1)$ and $\pi(S_2)$;
\item $F_2$, if irreducible,  maps $2:1$ onto a conic.
\end{enumerate}
\end{proposition}

\begin{proof}
We see immediately that
$$D^2=(3F+S_1+S_2)^2=6 FS_1 + 6 F S_2 +S_1^2 + S_2^2 = 12 - 3 -3 = 6.
$$
Therefore, $\chi(D)=\frac 12 D(D-F) + \chi(\Oo_X) = \frac 12 (6 - 2) + 2 = 4$.

\smallskip
Consider the short exact sequence
\begin{equation}\label{3F_D_ses}
0\to \Oo(3F) \to \Oo(D) \to \Oo(D)\vert_{S_1}\oplus \Oo(D)\vert_{S_2}\to 0.
\end{equation}
We know that the bicanonical map of $X$ is the elliptic fibration and has $\Pp^1$ as its image. Thus the canonical ring of $X$ is a polynomial ring with generators of weights $1$ and $2$, corresponding to $F$ and $F_2$,  so
$h^0(3F)=2$. Together with the Euler characteristics calculation and $h^2(3F)=h^0(-2F)=0$ this implies that
$h^1(3F)=0$.

\smallskip
Because of $(3F+S_i)S_i=3-3=0$, we know that the restrictions of the sheaf  $\Oo(D)$ to either $S_i$ is isomorphic to the structure sheaf.
Thus the long exact sequence in cohomology associated to \eqref{3F_D_ses} implies that $\dim H^0(X,\Oo(D))= 2 +1+ 1 = 4$. The same long exact sequence implies $h^1(D)=h^2(D)=0$.

\smallskip
Let us now prove that $H^0(X,\Oo(D))$ is base point free.
The long exact sequence associated to \eqref{3F_D_ses} implies that there are sections which restrict to non-zero constants on $S_1$ and $S_2$,
and the base locus of $H^0(X,\Oo(D))$ is contained in that of  $H^0(3F)$. We know that this space is generated by the sections with divisors $3F$ and $F+F_2$. Therefore, the base locus of $H^0(X,\Oo(D))$ is contained in $F$.
Consider the short exact sequence
$$
0\to \Oo(2F+S_1+S_2) \to \Oo(D) \to \Oo(D)\vert_{F}\to 0.
$$
Since $S_i(2F+S_1+S_2)<0$, either $S_i$ is a base component of $|2F+S_1+S_2|$, hence $h^0(2F+S_1+S_2)=h^0(2F)=2$. Since $h^2(2F+S_1+S_2)=0$, Riemann-Roch implies that $h^1(2F+S_1+S_2)=0$ and $h^0(\Oo(D)\vert_{F})=2$.  If $F$ is irreducible, then it is an elliptic curve and the restriction of $\Oo(D)$ to $F$ is the full linear system of degree two, hence is base point free, which implies that so is  $H^0(X,\Oo(D))$. Furthermore $F$ maps $2:1$ onto a line, which passes through the points $\pi(S_1)$ and $\pi(S_2)$. If $F$ is of type $I_2$, i.e., $F=R_1+R_2$ for two $(-2)$-curves $R_i$, then the restriction of $\Oo(D)$ to either $R_i$ is the full linear system of degree one, hence it is base point free and so is $H^0(X,\Oo(D))$, and $R_i$ maps $1:1$ onto a line $L_i$. Since $R_1$ and $R_2$ intersect at two distinct points, we see that $L_1=L_2$, but then $S_1$ must intersect both $R_1$ and $R_2$, contradicting $S_1F=1$.

Looking at the intersection number of $D$ with each component of the $I_9$-fibers, we easily see that each $I_9$-fiber maps to a union of a conic and two lines. Since the image of a fiber is contained in a hyperplane section of $\pi(X)$,  the degree of $\pi(X)$ is at least 4, hence must be 6.

%The calculation $D^2=6$ implies that either the image is a sextic, as claimed, or $H^0(X,\Oo(D))$ gives a triple cover of a quadric, or it gives a double cover of a cubic. We assume the most likely option of a sextic, which is  later confirmed by calculations.

The restriction of $D$ to a general smooth fiber $H$ of $X\to \Pp^1$ gives the short exact sequence
$$
0\to \Oo(F+S_1+S_2) \to \Oo(D) \to \Oo(D)\vert_{H}\to 0.
$$
The corresponding long exact sequence shows that $H^0(X,\Oo(D))$ restricts to a 3-dimensional linear subspace of the 4-dimensional space of sections of a degree four line bundle on the elliptic curve $H$. The corresponding $\Pp^2$ contains the
line which is the image of $F$. The images of the fibers $H$ are degree four curves in $\Pp^2$ of genus one, unless they are double covers of conics. In the latter case, $\pi(X)$ would have degree $<6$, a contradiction.
%As a consequence, we may rule out the possibility of $X\to \Cc\Pp^3$ being a triple cover of a quadric, because in this case $H$ would need to map to a line.
%It remains to rule out the possibility of a double cover of a cubic.
%{\bf to be continued}

 Assume that $F_2$ is irreducible. Then it is an elliptic curve and the corresponding long exact sequence shows that $H^0(X, D)$ restricts to a 3-dimensional linear subspace of the 4-dimensional space $H^0(F_2, D\vert_{F_2})$.
 % of sections of a degree four line bundle on $F_2$.
 Let $a, a'$, possibly $a=a'$,  be the intersection points of $F_2$ and $S_1$. Then $F_2\cap S_2=\{a+t, a'+t\}$ for a fixed 2-torsion point $t\in F_2$, because the deck transformation of $X$ acts freely on $F_2$ and switches $S_1$ and $S_2$.
 %Since any pair of points on an elliptic curve defines a double cover of $\Pp^1$, we may assume that $a'=-a$.
  Let $\nu: F_2\to \Pp^1$ be the double cover given by the degree two linear system $|a+a'|$ on $F_2$. We claim that
 $$H^0(X, D)\vert_{F_2}=\nu^*H^0(\Pp^1,\Oo(2))$$
 as 3-dimensional subspaces of $H^0(F_2, D\vert_{F_2})$.  To prove this, consider the subspace $$W_1:=H^0(X, 3F+S_1)\times H^0(S_2)\subset H^0(X, D).$$ Since  $h^0(3F)=2$, $h^1(3F)=0$, and $\Oo(3F+S_1)$ restricts to the structure sheaf of $S_1$, we see that $h^0(3F+S_1)=3$, hence $\dim W_1=3$. It is easy to compute  $h^0(F+S_1)=1$, $h^1(F+S_1)=0$, which implies that $H^0(3F+S_1)$ restrict to the full linear system of $H^0(F_2, (3F+S_1)\vert_{F_2})$. The latter space equals the 2-dimensional space $H^0(F_2, a+a')$ of the degree two line bundle $\Oo_{F_2}(a+a')$. Thus $W_1$ restricts to the 2-dimensional space corresponding to the 1-dimensional linear system $|a+a'|+(a+t)+(a'+t)$.
 Since $(a+t)+(a'+t)\in |a+a'|$, this 1-dimensional linear system belongs to the linear system of $\nu^*H^0(\Pp^1,\Oo(2))$. Similarly, $W_2:=H^0(X, 3F+S_2)\times H^0(S_1)\subset H^0(X, D)$ restricts to the 2-dimensional space corresponding to the linear system $a+a'+|(a+t)+(a'+t)|$. The two 2-dimensional spaces $W_i\vert_{F_2}$ in $\nu^*H^0(\Pp^1,\Oo(2))$ has  1-dimensional intersection, which corresponds to the unique divisor $a+a'+(a+t)+(a'+t)$. The claim and the last  assertion is proved.
\end{proof}

We remark that $F_2$, if reducible, maps onto a union of two conics.

\begin{remark}
We eventually expected that a fake projective plane can be identified as such once we have its explicit equations, as we did in Section 2. As a consequence, we felt free to pursue the most likely scenarios rather than try to exhaustively exclude all degenerate cases, since the justification of our approach will be in its final result. This liberating philosophy is similar to the physicists' approach to mathematics: anything goes as long as the final answers concur with experiments. In particular, we assume that $F_{2}$ is irreducible.
%we assume that $F$ and $F_{2}$ are irreducible.
%To make it clear that our proofs don't necessarily take into account all possible subtleties, we will use the name Observation, in place of Proposition.
\end{remark}

\section{Breakthrough: equation of the image of the double cover $X$}
 %4. Equation of the image of the double cover of resolution of FPP/Z^7. Notable curves on the double cover.
%   Normalization of the hypersurface.
In this section we describe the major breakthrough that allowed us to eventually write down the equations of the fake projective plane.
Specifically, we describe the method that allowed us to find the  $\Zz_2$-invariant sextic in $\Cc\Pp^3$ which gives
a (highly singular) birational model of the double cover $X$ of the resolution of the $\Zz_7$-quotient.

\smallskip
The action of the covering involution $\sigma$ on $X$ leads to an involution on $H^0(X,D)$ which has two-dimensional eigenspaces.
We observe that there are two natural, up to scaling, elements  $y_0$ and $y_1$ of $H^0(X,\Oo(D))$ which correspond to divisors
$$
F+F_2 + S_1 +S_2,~~3F+S_1+S_2
$$
respectively. We will linearize the action of the covering involution $\sigma$ so that $\sigma(y_0)=y_0$ and $\sigma(y_1)=-y_1$.
We pick other basis elements of the eigenspaces and denote them by $y_2$ and $y_3$.

\smallskip
We know that the images of $S_1$ and $S_2$ are disjoint points on $(0:0:*:*)$ which are permuted by the involution. We can scale
$y_2$ and $y_3$ to ensure that these are $(0:0:-1:1)$ and $(0:0:1:1)$ respectively. For generic $a$ the divisor of $y_0-a y_1$ is
$F_1+S_1+S_2+H_a$ where $H_a$ is a fiber of $X\to\Cc\Pp^1$. Note that $S_1$ and $S_2$ intersect $H_a$ in two points each. These points
need to map to the the same point in $\Cc\Pp^3$ which leads to the statement  in Proposition \ref{crucial} that the image of $H_a$ is a nodal plane quartic with two nodes.

\smallskip
We also know that $F_2$ maps $2:1$ onto a conic.

\smallskip
Putting it all together, the geometry of $\pi:X\to \Cc\Pp^3$ implies the following.

\begin{itemize}
\item The involution acts by $y_i \mapsto (-1)^i y_i$. The sextic $f$ is invariant with respect to this involution.

\item The sections $y_0$ and $y_1$ are zero on $S_1$ and $S_2$. These are automatically zero on $F$.

\item The section $y_1=0$ corresponds to the divisor $3F + S_1 +S_2$ and the section $y_0=0$ corresponds to $F + F_2 +S_1 +S_2$.
The image of $F$ is $(0:0:*:*)$. This is a $2:1$ cover, so $f=0$ has singularities along $(0:0:*:*)$.

\item For $a\neq 0$ the restriction of $f$ to $x_0=a x_1$ is
$$f(ay_1,y_1,y_2,y_3) = y_1^2 g_a(y_1,y_2,y_3)$$
where $g_a=0$ is a degree four curve which has nodes
at $(0:\pm 1:1)$.

\item For $a\neq 0$ the quartic $g_a=0$ is irreducible, except for $a=\pm 1$ that correspond to the images of the $I_9$ fibers. (We can fix $a=\pm 1$ for the location of $I_9$ fiber by scaling $y_0$ and $y_1$.)

\item The restriction to $y_1=0$ is given by
$$
f(y_0,0,y_2,y_3) = y_0^6.
$$
Indeed, we must have a multiple of $F_1$ (since $S_1$ and $S_2$ map to points). This means that this should be a multiple of $y_0$
and we can scale it to be $y_0^6$.

\item The restriction of $f$ to $y_0=0$ is
given by
$$
f(0,y_1,y_2,y_3) = y_1^2 h_0^2(y_1,y_2,y_3)
$$
where $h_0=0$ is a $\sigma$-invariant conic that passes through $(0:\pm 1:1)$. The surface $f=0$
has singularities along $(0:h_0=0)$.
\end{itemize}

There are additional restrictions on $f=0$ that come from the geometry of the $I_9$ fibers.
Without loss of generality, let us assume that the fiber at $y_0=y_1$ corresponds to the image of the cycle of curves
$$
A_1-B_1-C_1-A_1'-B_1'-C_1'-A_1''-B_1''-C_1''-A_1
$$
and the $y_0=-y_1$ fiber corresponds to the cycle $A_2-\ldots-C_2''-A_2$.

\medskip
The intersection numbers \label{intersections_on_X} imply that
$$
DA_1'=2,~ DA_1''= 1,DB_1=1
$$
so the degree four genus one curve with two nodes degenerates into two lines $\pi(A_1'')$ and $\pi(B_1)$ and a conic $\pi(A_1')$.
The other six rational curves of this $I_9$ fiber are contracted to singular points. The line $\pi(B_1)$ must pass through $\pi(S_1)=(0:0:-1:1)$
as does the conic $\pi(A_1')$. The line $\pi(A_1'')$ passes through the other node $\pi(S_2)=(0:0:1:1)$. These lines intersect at some point $P$
which we can set to be $P=(1:1:0:0)$ by adding multiples of $y_0$ and $y_1$ to $y_2$ and $y_3$ respectively. Moreover we see
that
$$
P=\pi(B_1'')=\pi(C_1'')=\pi(A_1)
$$
so the surface $\pi(X)$ has at least an $A_3$ type singularity at $P$. In particular, the partial derivatives and the derivative of the Hessian matrix vanish at $P$. In addition, we have a singular point  $\pi(C_1)$ at the intersection of the line $\pi(B_1)$ and the conic $\pi(A_1')$ which is different from
$\pi(S_1)=(0:0:-1:1)$. We also have a singular point
$$
\pi(B_1')=\pi(C_1')
$$
at the intersection of the conic $\pi(A_1'')$ and the line $\pi(A_1'')$ that is different from $\pi(S_2)=(0:0:1:1)$.

\medskip
We immediately see from the intersection numbers that
$$
D S_1' = D S_2' =  DS_1''=DS_2''=3.
$$
We focus specifically on $S_1''$. Note that $S_1''$ intersects both $B_1''$ and $A_1$, which means that $\pi(S_1'')$ passes through $(1:1:0:0)$ twice.
Thus it should be a planar degree three rational nodal curve. This turned out to be a key observation that allowed us to get enough equations on the
coefficients of $f$ to solve for it.

\begin{proposition}\label{FyS}
The sextic equation $f(y_0,y_1,y_2,y_3)=0$ where
$$
f= 28 y_0^6 -(42 - 2 \ii \sqrt{7}) y_0^4 y_1^2 - 4 \ii \sqrt{7} y_0^2 y_1^4 +
 56 y_0^2 y_1^2 y_2^2 - (14 + 22 \ii  \sqrt{7}) y_0^4 y_1 y_3
 $$
 $$-(7 -
    13 \ii \sqrt{7}) y_0^2 y_1^2 y_3^2 - (77 +
    17 \ii \sqrt{7}) y_1^4 y_3^2 + (21 - 31 \ii \sqrt{7}) (y_0^3 y_1 y_2 y_3 -
    y_0 y_1^3 y_2 y_3)
$$
$$
-(28 - 20 \ii \sqrt{7}) y_1^3 y_3 (y_1^2 + y_2^2 -
    y_3^2) + (14 + 2 \ii \sqrt{7}) y_1^2 (y_1^4 +
    2 y_1^2 y_2^2 + (y_2^2 - y_3^2)^2)
$$
$$
 + (42 +
    2 \ii \sqrt{7}) (y_0^2 y_1^3 y_3 +
    y_0 y_1^2 y_2 (-y_0^2 + y_1^2 + y_2^2 - y_3^2))
$$
cuts out a surface which has the same expected properties as the image of the double cover $X$ under the map
given by $\vert 3F+S_1+S_2\vert$.
\end{proposition}

\begin{remark}
It is clear that complex conjugation provides another surface with the same properties that comes from the complex conjugate fake projective plane.
\end{remark}

\medskip
We remark that the formula of Proposition \ref{FyS} was obtained by writing down a generic invariant sextic that satisfied the properties and then using Mathematica software package to write down equations on the coefficients. The equations are too complicated to be solved symbolically, but numerical solutions give values
that "look like" algebraic numbers. This allow us to identify a putative equation, which can then be checked to give the desired properties.

\medskip
We now describe the images of the $24$ curve $S_1,\ldots, S_2'',A_1,\ldots,C_2''$ on $\pi(X)$.
The curve $S_1''$ was found in the process of getting Proposition \ref{FyS}. The curve $S_2''$ is obtained by
simply applying the involution $\sigma$. It took a bit of effort to find $S_1'$. The idea is that there should be an order
three automorphism that acts fiberwise on $X\to \Cc\Pp^1$ and sends $S_1\to S_1'\to S_1''$. { This automorphism 
is a lift of the order 3 automorphism acting on the quotient $\Pp_{\fake}^2 /\Zz_7$.
}
Each of the curves
$S_1$, $S_1'$ and $S_1''$ have two points in the generic fiber, which give two orbits under addition of an element of order three.
Thus, if we parameterize $S_1''$ as $S_1''(t)$ there should be a point $S_1'(t)$ in the fiber so that
$$
S_1'(t) + S_1''(t) = (S_1)_1+(S_1)_2
$$
where $(S_1)_i$ are two preimages of the node $\pi(S_1)$. Since the preimage of the class of the line in $\Cc\Pp^2$ that contains the fiber is
$ (S_1)_1+(S_1)_2+ (S_2)_1+(S_2)_2$ we see that the fourth intersection point of the line through the node $\pi(S_2) = (0:0:1:1)$ and $S_1''(t)$ with
the quartic image of the fiber should
give parameterization of $S_1'$. We write the corresponding equations in Table \ref{curves}.

\begin{table}[htp]
\caption{Images of curves on $X$ under the map $\pi:X\to \Cc\Pp^3$
(The equations are either parametric or non-parametric; the curves $\pi(S_2),\ldots, \pi(C_2'')$ can be found by
applying $\sigma(y_0:y_1:y_2:y_3)=(y_0:-y_1:y_2:-y_3)$ to $\pi(S_1),\ldots, \pi(C_1'')$.)}
\begin{center}
$$
\begin{array}{|c|l|}
\hline
{\rm Curves} & {\rm Equations}\\[.2em]
\hline\hline
\pi(F)& y_0=y_1=0\\[.5em]
\hline
\pi(F_2)& y_0=0,~y_1^2 + y_2^2 + \frac 14 (-1 + 3 \ii \sqrt{7}) y_1 y_3 - y_3^2=0
\\[.5em]
\hline
\pi(S_1)& (0:0:-1:1)\\[.5em]
\hline
\pi(S_1')&
y_0 =  \frac 18 (11 - \ii \sqrt{7})t  + \frac 18(-3  +\ii\sqrt{7}) t^3
\\[.2em]
&y_1=t^3
\\[.2em]
&
y_2 =  \frac 18 (11 - \ii \sqrt{7}) + \frac 18 (-1 + 3 \ii\sqrt{7}) t -
\frac 18 ( 5 +\ii \sqrt{7}) t^2 + \frac 18 (3 - \ii \sqrt{7}) t^3,
\\[.2em]
&
y_3=-\frac 1{16}(9 + 5 \ii \sqrt{7})+
    \frac 1{16} (11 - \ii \sqrt{7})t + \frac 1{16} (21 + \ii\sqrt{7})t^2 - \frac 1{16}(7 - 5 \ii \sqrt{7}) t^3
\\[.5em]
\hline
\pi(S_1'')&
y_0 =  \frac 18 (11 - \ii \sqrt{7})t  + \frac 18(-3  +\ii\sqrt{7}) t^3
\\[.2em]
&
y_1=t^3
\\[.2em]
&
y_2= \frac 1{16} {(-9 - 5 \ii \sqrt{7} + (11 - \ii \sqrt{7}) t) (-1 + t^2)}
\\[.2em]
&
y_3= \frac 18{(11 - \ii \sqrt{7} +(- 1 + 3 \ii \sqrt{7}) t) (-1 + t^2)}
\\[.5em]
\hline
\pi(A_1)&
(1:1:0:0)
\\[.5em]
\hline
\pi(B_1)&
y_0 = y_1,~y_2=-y_3
\\[.5em]
\hline
\pi(C_1)&
(1:1:
-\frac 14  (3 +\ii \sqrt{7}) , \frac 14  (3 +\ii \sqrt{7}) )
\\[.5em]
\hline
\pi(A_1')&
y_0=y_1,
\\[.2em]&
\frac 12(11 - \ii \sqrt{7}) y_1^2 + \frac 14 (11 - \ii \sqrt{7}) y_1 y_2 + y_2^2 +
 \frac 12 (-1 + 3 \ii \sqrt{7}) y_1 y_3 - y_3^2=0
\\[.5em]
\hline
\pi(B_1')&
(-1:-1:\frac 12 (1 - \ii \sqrt{7}):\frac 12 (1 - \ii \sqrt{7}))
\\[.5em]\hline
\pi(C_1')&
(-1:-1:\frac 12 (1 - \ii \sqrt{7}):\frac 12 (1 - \ii \sqrt{7}))
\\[.5em]\hline
\pi(A_1'')&
y_0=y_1,~y_2=y_3
\\[.5em]\hline
\pi(B_1'')&
(1:1:0:0)
\\[.5em]\hline
\pi(C_1'')&
(1:1:0:0)
\\[.5em]
\hline
\end{array}
$$
\end{center}
\label{curves}
\end{table}

\begin{remark}\label{cones}
The construction of $S_1'$ has an additional advantage of providing us with a rational function on $Y$ which has well-understood zeros and poles.
Specifically, the section
$$
\Big(
y_0^3 - y_0^2 y_1 - y_0 y_1^2 + y_1^3 +
 \frac 12 (1 + \ii \sqrt{7}) (y_0-y_1) y_1 (y_2 - y_3)
 +
 \frac 14  (-1 + \ii\sqrt{7}) y_1 (y_2 - y_3)^2
 \Big)
$$
defines a (nonnormal) cubic cone with vertex $(0:0:1:1)$ that contains $S_1'$ and $S_1''$. Its symmetrization $f_{cones}(y_0,y_1,y_2,y_3)$
given by
$$
\Big(
y_0^3 - y_0^2 y_1 - y_0 y_1^2 + y_1^3 +
 \frac 12 (1 + \ii \sqrt{7}) (y_0-y_1) y_1 (y_2 - y_3)
 +
 \frac 14  (-1 + \ii\sqrt{7}) y_1 (y_2 - y_3)^2
 \Big)
$$
$$
\Big(
y_0^3 + y_0^2 y_1 - y_0 y_1^2 - y_1^3 -
 \frac 12 (1 + \ii \sqrt{7}) (y_0+y_1) y_1 (y_2 + y_3)
 -
 \frac 14  (-1 + \ii\sqrt{7}) y_1 (y_2 + y_3)^2
 \Big)
$$
gives a $\sigma$-invariant section of $H^0(X,6D)$ which contains $S_1'+S_2'+S_1''+S_2''$. In fact, we were able to
show that its degree $36$ intersection curve with $\pi(X)$ is fully accounted for by the curves from our list of $24$ rational curves
as well as $F$. The $\sigma$-invariant rational function
$$
\frac {f_{cones}(y_0,y_1,y_2,y_3)}{(y_0^2-y_1^2)^{3}}
$$
on $X$ gives a rational function on $Y$ whose divisor is
$$
2A-3A'+A'' - B - B' +2B'' -2C+2C''-2S+S'+S''.
$$
\end{remark}

\medskip
The curves $A_1,\ldots,C_1''$ are either contracted  to points or map isomorphically to lines or conics  in the plane $y_0=y_1$, as indicated in
Table \ref{curves}.

\medskip
An important part of our calculations will be based on finding a putative normalization of the ring
$$
\Cc[y_0,y_1,y_2,y_3]/\la f(y_0,y_1,y_2,y_3)\ra.
$$
\begin{proposition}\label{normalization}
The rational functions
$$
\hat y_4 = \frac {y_0^3}{y_1}
$$
$$
\hat y_5 = \frac {(y_1^2 + y_2^2 + \frac 14 (-1 + 3 \ii \sqrt{7}) y_1 y_3 - y_3^2) y_1}{y_0}
$$
lie in the normalization of  $\Cc[y_0,y_1,y_2,y_3]/\la f(y_0,y_1,y_2,y_3)\ra$ in its field of fractions. These elements are odd with respect to the involution
$\sigma$ and are homogeneous with grading $2$.
\end{proposition}

\begin{proof}
It is straightforward to see that $\hat y_4$ and $\hat y_5$ satisfy monic quadratic equations with coefficients in the ring. The parity and grading are obvious.
\end{proof}

\begin{remark}
We believe that $y_0,\ldots,y_3,\hat y_4,\hat y_5$ generate the normalization of the ring  $\Cc[y_0,y_1,y_2,y_3]/\la f(y_0,y_1,y_2,y_3)\ra$
which is isomorphic to
$$
\bigoplus_{k\geq 0} H^0(X,\Oo(kD)).
$$
Moreover, we have calculated generators of the ideal of relations between $y_0,\ldots,y_3, \hat y_4, \hat y_5$. Since we do not need this information for our purposes, we will not present it in the paper. However, we do use the fact that $\hat y_4$ and $\hat y_5$ give odd sections of $H^0(X,\Oo(2D))$.
\end{remark}

\section{Order three automorphism.}
%5. Order three automorphism of the double cover.
An important feature of $X$ is an order three automorphism { which is  a lift of the order 3 automorphism acting on the quotient $\Pp_{\fake}^2 /\Zz_7$.} In this section we describe how to find an explicit formula for it in terms of the
birational automorphism of the sextic surface $\pi(X)\subset \Cc\Pp^3$.

\begin{proposition}\label{auto}
Let $Y_0=\frac {y_0}{y_1}$, $Y_2=\frac {y_2}{y_1}$ and $Y_3=\frac {y_3}{y_1}$ be the generators of the field extension ${\mathrm {Rat}}(X)\supset \Cc$.
The automorphism of order three sends $(Y_0,Y_2,Y_3)$ to $(Y_0, Y_2',Y_3')$ given by Table \ref{auto.table}. Its inverse
sends  $(Y_0,Y_2,Y_3)$ to $(Y_0, Y_2'',Y_3'')$ given by Table \ref{auto.table.inverse}.
\end{proposition}

\begin{table}[htp]
\caption{Automorphism of order 3 $:(Y_0, Y_2, Y_3)\mapsto (Y_0, Y_2', Y_3')$}
\begin{center}
$$
\begin{array}{|ll|}
\hline
Y_2' = &
\frac {(3 +
     \ii \sqrt{7})}8
     \, Y_0^{-1} ((-21 \ii + 31 \sqrt{7}) Y_2^2 + (-35 \ii + 9 \sqrt{7}) Y_3^2))^{-1}
     \\[.4em]&
      \Big(7 (9 \ii + 5 \sqrt{7}) Y_0^4 Y_3 +
     2 Y_0^2 (4 (21 \ii + \sqrt{7}) Y_2^2 - (7 \ii + 11 \sqrt{7}) Y_3) +
     \\[.4em]&
     Y_3 (-49 \ii - 13 \sqrt{7} - (49 \ii + 13 \sqrt{7}) Y_2^2 +
        8 (-7 \ii + 5 \sqrt{7}) Y_3 + (49 \ii + 13 \sqrt{7}) Y_3^2)
         \\[.4em]&   -
     Y_0 ((-21 \ii + 31 \sqrt{7}) Y_2^3 +
        Y_2 Y_3 (112 \ii + 48 \sqrt{7} + 21 \ii Y_3 - 31 \sqrt{7} Y_3))\Big)
   \\[.4em]
Y_3'= &\frac { (-3 \ii + \sqrt{7})}8
   Y_0^{-1} ((-21 \ii + 31 \sqrt{7}) Y_2^2 + (-35 \ii +
         9 \sqrt{7}) Y_3^2))^{-1}
       \\[.4em]&
           \Big((-21 - 31 \ii \sqrt{7}) Y_2^3 +
     Y_0 Y_2^2 (-168 + 8 \ii \sqrt{7} + 49 Y_3 - 13 \ii \sqrt{7} Y_3)
         \\[.4em]&
     +
     Y_0 Y_3^2 (56 + 40 \ii \sqrt{7} - 49 Y_3 + 13 \ii \sqrt{7} Y_3) +
     Y_2 (-21 - 31 \ii \sqrt{7}
         \\[.4em]&
          + 7 (13 + 7 \ii \sqrt{7}) Y_0^4 +
        8 (21 - \ii \sqrt{7}) Y_3 + (21 + 31 \ii \sqrt{7}) Y_3^2
            \\[.4em]&
             +
        Y_0^2 (-70 - 18 \ii \sqrt{7} + (-56 - 40 \ii \sqrt{7}) Y_3))\Big)
            \\[.8em]
            \hline
\end{array}
$$
\end{center}
\label{auto.table}
\end{table}

\begin{table}[htp]
\caption{Inverse automorphism of order 3 $:(Y_0, Y_2, Y_3)\mapsto (Y_0, Y_2'', Y_3'')$}
\begin{center}
$$
\begin{array}{|ll|}
\hline
Y_2'' = &\Big(-20 - 4 \ii \sqrt{7} - 4 \ii (-9 \ii + \sqrt{7}) Y_0^5 Y_2 +
     (34  - 30 \ii \sqrt{7} )Y_3 + (134  + 14 \ii \sqrt{7} )Y_3^2
     \\[.4em]&
      -  (15 - 43 \ii \sqrt{7}) Y_3^3 - 48 Y_3^4
      - (1+ 3 \ii \sqrt{7} )
     Y_3^5 + 4 \ii Y_0^6 (5 \ii - \sqrt{7} + 2 \sqrt{7} Y_3)
       \\[.4em]&
       +
     Y_2^4 (-20 - 4 \ii \sqrt{7} + (-1 - 3 \ii \sqrt{7}) Y_3) +
     2 Y_0^3 Y_2 (36 +
        4 \ii \sqrt[
         7] + (3 + 5 \ii \sqrt{7}) Y_2^2
    \\[.4em]&
           + (-5 + 15 \ii \sqrt{7}) Y_3 +
         (-16 +2\ii  \sqrt{7}) Y_3^2) +
     Y_2^2 (-40 - 8 \ii \sqrt{7} +
        33 (1 - \ii \sqrt{7}) Y_3
         \\[.4em]&
               + (68 + 4 \ii \sqrt{7}) Y_3^2 + (2 +
           6 \ii \sqrt{7}) Y_3^3) +
     2 Y_0^2 (10 + 2 \ii \sqrt{7} +
        8 Y_2^4 + (-26 + 10 \ii \sqrt{7}) Y_3
               \\[.4em]&
         + (-29 -
           9 \ii \sqrt{7}) Y_3^2 + (8 - 4 \ii \sqrt{7}) Y_3^3 -
        Y_2^2 Y_3 (17 + \ii \sqrt{7} + 8 Y_3)) +
     Y_0^4 (20 + 4 \ii \sqrt{7}
            \\[.4em]&
            + 2 (9 + \ii \sqrt{7}) Y_3 +
        4 (3 - \ii \sqrt{7}) Y_3^2 + (7 + 5 \ii \sqrt{7}) Y_3^3 +
        Y_2^2 (-48 + 16 \ii \sqrt{7} - (7+ 5 \ii \sqrt{7}) Y_3))
               \\[.4em]& +
     Y_0 Y_2 (-36 - 4 \ii \sqrt{7} + (5 - \ii \sqrt{7}) Y_2^4 +
        10 (1 - 3 \ii \sqrt{7}) Y_3 + 52 Y_3^2 + (5 - \ii \sqrt{7}) Y_3^4
               \\[.4em]&+
        2 \ii Y_2^2 (13 \ii -
           7 \sqrt{7} + (5 \ii + \sqrt{7}) Y_3^2))
           \Big)/\Big(2 Y_0 (-3 \ii - \sqrt{7} + (3 \ii + \sqrt{7}) Y_0^2 - 2 \ii Y_2^2
             \\[.4em]& + (- 5 \ii + \sqrt{7} )Y_3 +
       (\ii  + \sqrt{7} )Y_3^2 -
       Y_2 (-5 \ii + \sqrt{7} + (-\ii + \sqrt{7}) Y_3))
       \\[.4em]&
        (-3 \ii - \sqrt{7}
        + (3 \ii + \sqrt{7}) Y_0^2 - 2 \ii Y_2^2 - (5 \ii - \sqrt{7} )Y_3 +
       (\ii  + \sqrt{7}) Y_3^2 +
       Y_2 (-5 \ii + \sqrt{7}
       \\[.4em]& +
      (-\ii + \sqrt{7}) Y_3))\Big)
\\[.4em]
Y_3'' =& \Big(8 \ii \sqrt{7} Y_0^6 Y_2 +
     Y_0^5 (-40 - 8 \ii \sqrt{7} + (26 + 2 \ii \sqrt{7}) Y_3) +
     2 Y_0^3 (40 + 8 \ii \sqrt{7}
     \\[.4em]& +
        2 (-17 + 11 \ii \sqrt{7} + (1 - 2 \ii \sqrt{7}) Y_2^2) Y_3 + (-39 -
           11 \ii \sqrt{7}) Y_3^2 + (11 - 3 \ii \sqrt{7}) Y_3^3)
           \\[.4em]&
           +
     2 Y_0^2 Y_2 (-4 \ii \sqrt{7}+ (33 - 3 \ii \sqrt{7}) Y_3 + (25 + 9 \ii \sqrt{7}) Y_3^2 +
        4 \ii (\ii + \sqrt{7}) Y_3^3
        \\[.4em]& +
        4 Y_2^2 (-4 - \ii \sqrt{7} + (1- \ii \sqrt{7}) Y_3)) +
     Y_2 (8 \ii \sqrt{7} + (5 - \ii \sqrt{7}) Y_2^4 +
        2 \ii (27 \ii + \sqrt{7}) Y_3
        \\[.4em]&
         + (-23 - 17 \ii \sqrt{7}) Y_3^2 + (8 -
           8 \ii \sqrt{7}) Y_3^3 + (5 - \ii \sqrt{7}) Y_3^4 +
        Y_2^2 (5 + 7 \ii \sqrt{7} + 8 \ii (\ii + \sqrt{7}) Y_3
        \\[.4em]&
         +
           2 \ii (5 \ii + \sqrt{7}) Y_3^2)) +
     Y_0^4 ((7 - 3 \ii \sqrt{7}) Y_2^3 +
        \ii Y_2 (-8 \sqrt{7} +
           4 (3 \ii + \sqrt{7}) Y_3 + (7 \ii + 3 \sqrt{7}) Y_3^2))
           \\[.4em]& +
     Y_0 (-40 - 8 \ii \sqrt{7} + (42 - 46 \ii \sqrt{7}) Y_3 +
        2 (83 + 7 \ii \sqrt{7}) Y_3^2 + (-14 + 46 \ii \sqrt{7}) Y_3^3
        \\[.4em]& -
        48 Y_3^4 + (-1 - 3 \ii \sqrt{7}) Y_3^5 +
        Y_2^4 (-4 - 4 \ii \sqrt{7} + (-1 - 3 \ii \sqrt{7}) Y_3)
        \\[.4em]&
        +
        2 Y_2^2 (-44 +
           4 \ii \sqrt{7} + (-6 - 16 \ii \sqrt{7}) Y_3 + (26 +
              2 \ii \sqrt{7}) Y_3^2 + (1 +
              3 \ii \sqrt{7}) Y_3^3))\Big)/
              \\[.4em]&
              \Big(2 Y_0 (-3 \ii - \sqrt{7}+ (3 \ii + \sqrt{7}) Y_0^2 - 2 \ii Y_2^2 - (5 \ii - \sqrt{7} )Y_3 +
       (\ii  + \sqrt{7}) Y_3^2
       \\[.4em]& -
       Y_2 (-5 \ii + \sqrt{7} + (-\ii + \sqrt{7}) Y_3)) (-3 \ii - \sqrt{7} + (3 \ii + \sqrt{7}) Y_0^2 - 2 \ii Y_2^2 - (5 \ii - \sqrt{7} )Y_3
       \\[.4em]&+
       (\ii + \sqrt{7} )Y_3^2 +
       Y_2 (-5 \ii + \sqrt{7} + (-\ii + \sqrt{7}) Y_3))\Big)\\[.8em]\hline
\end{array}
$$
\end{center}
\label{auto.table.inverse}
\end{table}

\begin{remark}
While formulas of Tables \ref{auto.table} and \ref{auto.table.inverse} are not particularly inspiring, they are far preferable to some other formulas for the
automorphism that we initially found.
\end{remark}

\begin{proof}
It is a straightforward computer calculation to check that the formulas provide automorphisms. However, it takes too long to verify that the cube of it
is identity symbolically. It is, however, trivial to do so heuristically by taking a random point on $\pi(X)$ calculated to high precision and iterating
the automorphism three times.

\smallskip
To find the automorphism we used the fact that $Y_2$ and $Y_3$ are rational functions with poles along $3F+S_1+S_2$. So their transforms
should be rational functions with poles along $3F+S_1'+S_2'$ and $3F+S_1''+S_2''$. We also know that $Y_2+Y_3$ is zero on $S_1$ and
$Y_2-Y_3$ is zero on $S_2$. This allows us to fix the transforms up to constants, which can then be recovered.
\end{proof}

\section{Double cover of the fake projective plane.}
%6. Sevenfold cover of the double cove: the field of the double cover of fake projective plane.
In this section we explain how we found the function field of the fake projective plane.

\medskip
According to \cite{K11} we need to attach the seventh root of the rational function which has divisor
$$
5S + B + 4C + 6S' + 4B' + 2C' + 3S''+2B''+C''
$$
up to multiples of $7$. { This divisor is divisible by 7 in the Picard group and corresponds to the third possibility for the divisor $B$ in [ibid., p 1676], where the curves $A_1, A_2, E_1, B_1, B_2, E_2, C_1, C_2, E_3$ correspond to $C'', B'', S'', C', B', S', C, B, S$ in our notation. The first possibility for $B$ was ruled out, because the $I_9$-fibre has multiplicity $\mu=1$  by [\cite{K17}, p 2 and Theorem 2.3 (5)], and the second possibility corresponds to Case 1 of \eqref{cases} which was ruled out in Section 3.
}
 We found this function by looking at the equation of the cubic cone with vertex $(0:0:1:1)$ that contains
$S_1''$ and $S_1'$. When divided by $y_1^3$ it
gives a divisor whose zeros and poles occur only at the named divisors. By symmetrizing it via $\sigma$ and using
the automorphism we were able to get the desired function. We denote the seventh root of this function by $z$; the function $z^7$ is given in
Table \ref{z.7}.

\begin{table}[htp]
\caption{Formula for $z^7$}
\begin{center}
$$
\begin{array}{|ll|}
\hline
z^7=&
\Big((-315\ii + 47\sqrt{7})^2 (-1 + Y_0^2)^5 (2795 \ii + 287 \sqrt{7} - 5590 \ii Y_0 -
      574 \sqrt{7} Y_0 + 11573 \ii Y_0^2 + 2689 \sqrt{7} Y_0^2
       \\[.4em]& -
      17556 \ii Y_0^3 - 4804 \sqrt{7} Y_0^3 + 14357 \ii Y_0^4 +
      5601 \sqrt{7} Y_0^4 - 11158 \ii Y_0^5 - 6398 \sqrt{7} Y_0^5
       \\[.4em]& +
      5579 \ii Y_0^6 + 3199 \sqrt{7} Y_0^6 + 5590 \ii Y_2^2
    +
      574 \sqrt{7} Y_2^2 - 5590 \ii Y_0 Y_2^2 - 574 \sqrt{7} Y_0 Y_2^2
       \\[.4em]& +
      5994 \ii Y_0^2 Y_2^2 - 510 \sqrt{7} Y_0^2 Y_2^2 + 5590 \ii Y_0^3 Y_2^2
      +
      574 \sqrt{7} Y_0^3 Y_2^2 + 2795 \ii Y_2^4 + 287 \sqrt{7} Y_2^4
       \\[.4em]&+
      1616 \ii Y_3 - 4336 \sqrt{7} Y_3 + 5568 \ii Y_0 Y_3
    +
      5824 \sqrt{7} Y_0 Y_3 + 3232 \ii Y_0^2 Y_3 - 8672 \sqrt{7} Y_0^2 Y_3
         \\[.4em]&  -
      448 \ii Y_0^3 Y_3 + 11584 \sqrt{7} Y_0^3 Y_3 - 9968 \ii Y_0^4 Y_3
     -
      4400 \sqrt{7} Y_0^4 Y_3 + 11584 \ii Y_0^2 Y_2 Y_3
        \\[.4em]&+
      64 \sqrt{7} Y_0^2 Y_2 Y_3 - 17600 \ii Y_0^3 Y_2 Y_3
       +
      5696 \sqrt{7} Y_0^3 Y_2 Y_3 + 1616 \ii Y_2^2 Y_3
      \\[.4em]&-
      4336 \sqrt{7} Y_2^2 Y_3 + 7184 \ii Y_0 Y_2^2 Y_3
      +
      1488 \sqrt{7} Y_0 Y_2^2 Y_3 - 17174 \ii Y_3^2 - 638 \sqrt{7} Y_3^2
       \\[.4em]&+
      11606 \ii Y_0 Y_3^2 - 5186 \sqrt{7} Y_0 Y_3^2 - 5994 \ii Y_0^2 Y_3^2
        +
      510 \sqrt{7} Y_0^2 Y_3^2 + 5994 \ii Y_0^3 Y_3^2
      \\[.4em]& -
      510 \sqrt{7} Y_0^3 Y_3^2 - 5590 \ii Y_2^2 Y_3^2
      -
      574 \sqrt{7} Y_2^2 Y_3^2 - 1616 \ii Y_3^3 + 4336 \sqrt{7} Y_3^3
        \\[.4em]& -
      7184 \ii Y_0 Y_3^3 - 1488 \sqrt{7} Y_0 Y_3^3 + 2795 \ii Y_3^4
     +
      287 \sqrt{7} Y_3^4) (2795 \ii + 287 \sqrt{7} + 5590 \ii Y_0
        \\[.4em]&+
      574 \sqrt{7} Y_0 + 11573 \ii Y_0^2 + 2689 \sqrt{7} Y_0^2
       +
      17556 \ii Y_0^3 + 4804 \sqrt{7} Y_0^3 + 14357 \ii Y_0^4
       \\[.4em]&+
      5601 \sqrt{7} Y_0^4 + 11158 \ii Y_0^5 + 6398 \sqrt{7} Y_0^5
       +
      5579 \ii Y_0^6 + 3199 \sqrt{7} Y_0^6 + 5590 \ii Y_2^2
       \\[.4em]&+
      574 \sqrt{7} Y_2^2 + 5590 \ii Y_0 Y_2^2 + 574 \sqrt{7} Y_0 Y_2^2
     +
      5994 \ii Y_0^2 Y_2^2 - 510 \sqrt{7} Y_0^2 Y_2^2 - 5590 \ii Y_0^3 Y_2^2
         \\[.4em]&-
      574 \sqrt{7} Y_0^3 Y_2^2 + 2795 \ii Y_2^4 + 287 \sqrt{7} Y_2^4
     +
      1616 \ii Y_3 - 4336 \sqrt{7} Y_3 - 5568 \ii Y_0 Y_3
        \\[.4em]&-
      5824 \sqrt{7} Y_0 Y_3 + 3232 \ii Y_0^2 Y_3 - 8672 \sqrt{7} Y_0^2 Y_3
     +
      448 \ii Y_0^3 Y_3 - 11584 \sqrt{7} Y_0^3 Y_3
         \\[.4em]&- 9968 \ii Y_0^4 Y_3
      -
      4400 \sqrt{7} Y_0^4 Y_3 - 11584 \ii Y_0^2 Y_2 Y_3
       -
      64 \sqrt{7} Y_0^2 Y_2 Y_3 - 17600 \ii Y_0^3 Y_2 Y_3
      \\[.4em]&+
      5696 \sqrt{7} Y_0^3 Y_2 Y_3 + 1616 \ii Y_2^2 Y_3
      -
      4336 \sqrt{7} Y_2^2 Y_3 - 7184 \ii Y_0 Y_2^2 Y_3
        \\[.4em]& -
      1488 \sqrt{7} Y_0 Y_2^2 Y_3 - 17174 \ii Y_3^2 - 638 \sqrt{7} Y_3^2
        -
      11606 \ii Y_0 Y_3^2 + 5186 \sqrt{7} Y_0 Y_3^2
      \\[.4em]&
       - 5994 \ii Y_0^2 Y_3^2 +
      510 \sqrt{7} Y_0^2 Y_3^2 - 5994 \ii Y_0^3 Y_3^2
       +
      510 \sqrt{7} Y_0^3 Y_3^2
      \\[.4em]&
      - 5590 \ii Y_2^2 Y_3^2 -
      574 \sqrt{7} Y_2^2 Y_3^2 - 1616 \ii Y_3^3 + 4336 \sqrt{7} Y_3^3
      +
      7184 \ii Y_0 Y_3^3
       \\[.4em]&
       + 1488 \sqrt{7} Y_0 Y_3^3 + 2795 \ii Y_3^4 +
      287 \sqrt{7} Y_3^4)\Big)/\Big(4096 Y_0^4 (-4 \ii
      +
      4 \ii Y_0
       \\[.4em]&+ 4 \ii Y_0^2 - 4 \ii Y_0^3 + 2 \ii Y_2 - 2 \sqrt{7} Y_2 -
      2 \ii Y_0 Y_2 + 2 \sqrt{7} Y_0 Y_2 + \ii Y_2^2 + \sqrt{7} Y_2^2 - 2 \ii Y_3
       \\[.4em]&+
      2 \sqrt{7} Y_3 + 2 \ii Y_0 Y_3 - 2 \sqrt{7} Y_0 Y_3 - 2 \ii Y_2 Y_3 -
      2 \sqrt{7} Y_2 Y_3 + \ii Y_3^2 + \sqrt{7} Y_3^2)^2 (-4 \ii - 4 \ii Y_0
       \\[.4em]&+
      4 \ii Y_0^2 + 4 \ii Y_0^3 - 2 \ii Y_2 + 2 \sqrt{7} Y_2 - 2 \ii Y_0 Y_2 +
      2 \sqrt{7} Y_0 Y_2 + \ii Y_2^2 + \sqrt{7} Y_2^2 - 2 \ii Y_3
       \\[.4em]&+
      2 \sqrt{7} Y_3 - 2 \ii Y_0 Y_3 + 2 \sqrt{7} Y_0 Y_3 + 2 \ii Y_2 Y_3 +
      2 \sqrt{7} Y_2 Y_3 + \ii Y_3^2 + \sqrt{7} Y_3^2)^2 (-21 \ii Y_2^2
       \\[.4em]& +
      31 \sqrt{7} Y_2^2 - 35 \ii Y_3^2 + 9 \sqrt{7} Y_3^2)^2\Big )
      \\[.8em]
      \hline
\end{array}
$$
\end{center}
\label{z.7}
\end{table}

\medskip
To find the function field of the fake projective plane we simply need to take the invariants with respect to $\sigma$ that preserves $z$ and $Y_3$ and
negates $Y_0$ and $Y_2$.

\medskip
We also found a lift of the action of the order three automorphism to the field generated by $Y_0,Y_2,Y_3,z$. Specifically, the action on $z$ is given in Table \ref{z.action}.

\begin{table}[htp]
\caption{Automorphism of order 3 $:(Y_0, Y_2, Y_3, z)\mapsto (Y_0, Y_2', Y_3', z'')$}
\begin{center}
$$
\begin{array}{|ll|}
\hline
z''=&z^2
(-1 + Y_0^2)^{-3} (1 - Y_0 - Y_0^2 + Y_0^3 -
\frac  12 ( 1+ \ii \sqrt{7}) (Y_2 - Y_3)
   \\[.4em]&
    +
   \frac 12 (1 + \ii \sqrt{7}) Y_0 (Y_2 - Y_3)
   +
   \frac 14 (-1 +\ii \sqrt{7}) (Y_2 - Y_3)^2) (-1 - Y_0 + Y_0^2 + Y_0^3
    \\[.4em]&
    -
   \frac 12  (1+ \ii \sqrt{7}) (Y_2 + Y_3)
       -
   \frac 12  (1 +\ii \sqrt{7}) Y_0 (Y_2 + Y_3) +
   \frac 14(1 - \ii \sqrt{7}) (Y_2 + Y_3)^2)
\\
\hline
\end{array}
$$
\end{center}
\label{z.action}
\end{table}

\section{Embedding of the fake projective plane into $\Cc\Pp^9$}
% 7. Constructing embeddings of fake projective plane. Specifically, 6H.
Let us now describe the method that allowed us to construct the equations of the fake projective plane.

\medskip
By a Riemann-Roch calculation, the dimension of the bicanonical linear system on  $\Pp^2_{\fake}$ is $10$.

\medskip
The pullback of the ($\Qq$-Cartier) canonical divisor via $\mu:Y\to\Pp^2_{\fake}/\Zz_7$ satisfies
\begin{equation}\label{pullback}
K_Y=\mu^* K_{\Pp^2_{\fake}/\Zz_7} -\frac 37 (S+S'+S'') - \frac 27 (B+B'+B'') - \frac 17 (C+C'+C'').
\end{equation}
This shows that the preimage of $\mu(F_Y)$ on $\Pp^2_{\fake}$ is numerically equivalent to a canonical divisor.
(It is actually a section of a canonical line bundle twisted by an invertible torsion line bundle). In particular, to
calculate
$$
H^0(\Pp^2_{\fake}, 2 K_{\Pp^2_{\fake}})
$$
we can look for rational functions on $\Pp^2_{\fake}$ which have poles of order at most two on the curve $F^{FPP}$ which is
the preimage of $\mu(F_Y)$ and no other poles.

\medskip
The action of $\Zz_7$ splits the space of such functions into seven eigenspaces. Each eigenspace consists of functions
of the form $z^i g$ where $g$ is a function from the function field of $Y$, as $i$ runs over residues modulo $7$.
The residual $\Zz_3$ action allows us to reduce the calculation to that of $i=-1,0,1$.

\medskip
The $i=0$ case is easy. The only such function up to scaling is $1$.

\medskip
Now let us calculate such functions of the form $z g$. Consider the Cartesian product diagram below
$$
\begin{array}{ccc}
\widehat{\Pp^2_{\fake} }&\to &Y\\
 \downarrow&&\downarrow~\\
\Pp^2_{\fake} &\to &\Pp^2_{\fake}/\Zz_7
\end{array}
$$
where $\widehat{\Pp^2_{\fake} }$ is the singular Galois cover of $Y$ ramified at the nine curves $S,\ldots, C''$ given
by normalization of $Y$ in the field of fractions of $\Pp^2_{\fake}$. We can calculate the global sections of an invertible sheaf
on $\Pp^2_{\fake}$ in terms of the pullback of these sections on $\widehat{\Pp^2_{\fake} }$.

\medskip
In view of \eqref{pullback} we see that the pullback of $2F^{FPP}$ on $\widehat{\Pp^2_{\fake} }$ is equal to twice its
proper preimage $\widehat{F^{FPP}}$ plus
$$
 \frac 67 (S+S'+S'')+ \frac 47 (B+B'+B'')  + \frac 27  (C+C'+C'').
$$
where $\frac 17 S$ is the reduced preimage of $S$ under $\widehat{\Pp^2_{\fake} }\to Y$, and similarly for the other eight curves.
The divisor of $z$ on  $\widehat{\Pp^2_{\fake} }$ is
$$
-A + A' + \frac  5 7 S  -\frac 1 7 S' -\frac 4 7 S''+ \frac 17 B+\frac 47 B' -\frac 57 B'' +\frac  47 C  +  \frac 27 C'  -\frac 67 C'' .
$$
This means that the divisor of $g$ on   $\widehat{\Pp^2_{\fake} }$ must be greater or equal to
$$
-2\widehat{F^{FPP}} - \frac 67 (S+S'+S'')- \frac 47 (B+B'+B'')  - \frac 27  (C+C'+C'')-{\mathrm div}(z)
$$
$$
=-2\widehat{F^{FPP}}+A-A' - \frac  {11} 7 S  -\frac 5 7 S' -\frac 2 7 S''- \frac 57 B -\frac 87 B' + \frac 17 B'' - \frac  67 C  -  \frac 47 C'  +\frac 47 C'' .
$$
Since $g$ is a rational function on $Y$, this translates into the condition that the divisor of $g$ on $Y$ is greater or equal than
$$
-2F_Y + A - A' - S -B' +B''+C'',
$$
in other words, it can be computed as a global section of the invertible sheaf
$$
\Oo_Y(2F_Y+S-A+A'+B'-B''-C'')
$$
on $Y$, or equivalently $\sigma$-invariant sections of
$$
\Oo_X(2F+S_1+S_2-A_1-A_2+A_1'+A_2'+B_1'+B_2'-B_1''-B_2''-C_1''-C_2'').
$$

\medskip
Note that the rational function $Y_0^2-1$ on $Y$
has pole of order $2$ at $F_Y$ and zeros of order $1$ at the nine curves $A,\ldots,C''$ of the $I_9$ fiber.
As a result, the $\sigma$-invariant section $y_0^2-y_1^2$ of $H^0(X,2D)$ is $2F+I_9+2S_1+2S_2$.
Since
$$
(2F_Y+I_9+2S) -(2F_Y -A+A'+B'-B''-C''+S) = S + 2A+A''+B+2B''+C+C'+2C'',
$$
we can find  $\sigma$-invariant sections of
$$
\Oo_X(2F+S_1+S_2-A_1-A_2+A_1'+A_2'+B_1'+B_2'-B_1''-B_2''-C_1''-C_2'').
$$
by looking at
$\sigma$-invariant sections of $2D$ which vanish on $(S + 2A+A''+B+2B''+C+C'+2C'')$. By using the calculation of Table \ref{curves}
it can be seen that such sections are multiples of $y_2^2-y_3^2$, so the rational function in question is
$$
\frac {(y_2^2-y_3^2)z}{y_0^2-y_1^2},
$$
up to a multiplicative constant.

%{\bf Maybe comment how the higher cohomology of the line bundle on  $\widehat{\Pp^2_{fake} }$ are zero, which means the same for the pushforward,
%thus giving that the dimension of the space of global sections coincides with the Euler characteristics, which is one.}

\medskip
Similarly, for the $ z^{-1}g$, we end up looking at $g$ which are global sections of
$$
\Oo_Y( 2F_Y+A-A'-C+B''+C''+S'+S'').
$$
We can construct these functions as
$$
\frac{(y_0^2-y_1^2)r(y_0,y_1,y_2,y_3)}{f_{cones}(y_0,y_1,y_2,y_3)}
$$
where $r(y_0,y_1,y_2,y_3)$ is a sigma-invariant section of
$
H^0(X,4D)
$
and $f_{cones}$ is given in Remark \ref{cones}.
{ 
The denominator $f_{cones}$ } is a $\sigma$-invariant element of $H^0(X,6D)$ which vanishes on $S'+S''$ given by
$$
\Big(
y_0^3 - y_0^2 y_1 - y_0 y_1^2 + y_1^3 +
 \frac 12 (1 + \ii \sqrt{7}) (y_0-y_1) y_1 (y_2 - y_3)
 +
 \frac 14  (-1 + \ii\sqrt{7}) y_1 (y_2 - y_3)^2
 \Big)
$$
$$
\Big(
y_0^3 + y_0^2 y_1 - y_0 y_1^2 - y_1^3 -
 \frac 12 (1 + \ii \sqrt{7}) (y_0+y_1) y_1 (y_2 + y_3)
 -
 \frac 14  (-1 + \ii\sqrt{7}) y_1 (y_2 + y_3)^2
 \Big).
$$
We know that the section $(y_0^2-y_1^2)$ of $H^0(Y,D)$ has divisor
$2F_Y+2S+I_9$ where $I_9=A+\ldots+C''$ is the sum of the curves in the $I_9$ fiber.
As a result,
the section $r$ should be vanishing on
$$
(4F_Y+4S+2I_9)+
(2A-3A'+A'' - B - B' +2B'' -2C+2C''-2S+S'+S'')
$$
$$
- (2F_Y+A-A'-C+B''+C''+S'+S'')
$$
$$
=2F_Y+2S+3A'+3A''+B+B'+3B''+C+2C'+3C''.
$$
Importantly, we need to use not just polynomial $r$ but also elements of the normalization, namely products of $\sigma$-antiinvaritant
degree two polynomials in $y_i$ with $\hat y_4$ and $\hat y_5$ from Proposition \ref{normalization}.

\medskip
This is a rather delicate calculation that led us to the results in Table \ref{R1R2}. Note that these functions are only determined up to linear changes
of variables. We have reduced the ambiguity a bit by requiring that the first of these sections vanishes at the fixed points of $\Zz_7$ action
on $\Pp^2_{\fake}$ and have chosen constants in a noble but not very successful attempt to make the equations more palatable.

\begin{table}[htp]
\caption{Rational functions $z^{-1}g$}
\begin{center}
$$
\begin{array}{|l|}
\hline
\Big(4 \ii (-1 + Y_0) (1 + Y_0) (-266 \ii Y_0 + 34 \sqrt{7} Y_0 + 532 \ii Y_0^3 -
      68 \sqrt{7} Y_0^3 - 266 \ii Y_0^5 + 34 \sqrt{7} Y_0^5
       \\[.4em]
        - 70 \ii Y_2 +
      46 \sqrt{7} Y_2 - 126 \ii Y_0^2 Y_2 - 58 \sqrt{7} Y_0^2 Y_2 +
      196 \ii Y_0^4 Y_2 + 12 \sqrt{7} Y_0^4 Y_2 - 469 \ii Y_0 Y_2^2
        \\[.4em]
        +
      97 \sqrt{7} Y_0 Y_2^2 - 63 \ii Y_0^3 Y_2^2 - 29 \sqrt{7} Y_0^3 Y_2^2 -
      70 \ii Y_2^3 + 46 \sqrt{7} Y_2^3 + 238 \ii Y_0 Y_3 + 266 \sqrt{7} Y_0 Y_3
      \\[.4em]
         -
      238 \ii Y_0^3 Y_3 - 266 \sqrt{7} Y_0^3 Y_3 + 259 \ii Y_2 Y_3 +
      41 \sqrt{7} Y_2 Y_3 - 259 \ii Y_0^2 Y_2 Y_3 - 41 \sqrt{7} Y_0^2 Y_2 Y_3
      \\[.4em]
+
      56 \ii Y_0 Y_2^2 Y_3 + 104 \sqrt{7} Y_0 Y_2^2 Y_3 + 728 \ii Y_0 Y_3^2 -
      56 \sqrt{7} Y_0 Y_3^2 - 196 \ii Y_0^3 Y_3^2 - 12 \sqrt{7} Y_0^3 Y_3^2
            \\[.4em]+
      70 \ii Y_2 Y_3^2 - 46 \sqrt{7} Y_2 Y_3^2 - 56 \ii Y_0 Y_3^3 -
      104 \sqrt{7} Y_0 Y_3^3)\Big)/\Big((-35 \ii + 23 \sqrt{7}) Y_0 (4 - 4 Y_0 -
      4 Y_0^2
            \\[.4em]
+ 4 Y_0^3 - 2 Y_2 -
      2 \ii \sqrt{7} Y_2 + (2 + 2 \ii \sqrt{7}) Y_0 (Y_2 - Y_3) + 2 Y_3 +
      2 \ii \sqrt{7} Y_3 +
      \ii (\ii + \sqrt{7}) (Y_2- Y_3)^2)
            \\[.4em]
(-4 \ii - 4 \ii Y_0 +
      4 \ii Y_0^2 + 4 \ii Y_0^3 - 2 \ii Y_0 Y_2 + 2 \sqrt{7} Y_0 Y_2 - 2 \ii Y_0 Y_3 +
      2 \sqrt{7} Y_0 Y_3 +  2 (-\ii + \sqrt{7})
            \\[.4em]
     (Y_2 + Y_3) + (\ii + \sqrt{7}) (Y_2 +
         Y_3)^2) z\Big)\\[.8em]\hline
\Big(16 \ii (-1 + Y_0) (1 + Y_0) (-133 \ii + 17 \sqrt{7} + 266 \ii Y_0^2 -
      34 \sqrt{7} Y_0^2
      - 133 \ii Y_0^4 + 17 \sqrt{7} Y_0^4
                  \\[.4em]
                  - 133 \ii Y_0 Y_2 +
      17 \sqrt{7} Y_0 Y_2
      + 133 \ii Y_0^3 Y_2 - 17 \sqrt{7} Y_0^3 Y_2 -
      217 \ii Y_2^2 + 37 \sqrt{7} Y_2^2 - 49 \ii Y_0^2 Y_2^2
      \\[.4em]
 -
      3 \sqrt{7} Y_0^2 Y_2^2 + 119 \ii Y_3 + 133 \sqrt{7} Y_3 -
      119 \ii Y_0^2 Y_3 - 133 \sqrt{7} Y_0^2 Y_3 + 217 \ii Y_3^2 -
      37 \sqrt{7} Y_3^2
      \\[.4em]
+ 49 \ii Y_0^2 Y_3^2 +
      3 \sqrt{7} Y_0^2 Y_3^2)\Big)/\Big((-35 \ii + 23 \sqrt{7}) (4 - 4 Y_0 -
      4 Y_0^2 + 4 Y_0^3 - 2 Y_2 -
      2 \ii \sqrt{7} Y_2
      \\[.4em]
+ (2 + 2 \ii \sqrt{7}) Y_0 (Y_2 - Y_3) + 2 Y_3 +
      2 \ii \sqrt{7} Y_3 +
      \ii (\ii + \sqrt{7}) (Y_2 -  Y_3)^2) (-4 \ii - 4 \ii Y_0 +
      4 \ii Y_0^2
      \\[.4em]
+ 4 \ii Y_0^3 - 2 \ii Y_0 Y_2
+ 2 \sqrt{7} Y_0 Y_2 - 2 \ii Y_0 Y_3 +
      2 \sqrt{7} Y_0 Y_3 +
      2 (-\ii + \sqrt{7}) (Y_2 + Y_3)
      \\[.4em]
+ (\ii + \sqrt{7}) (Y_2+
         Y_3)^2) z\Big)\\
         \hline
         \end{array}
$$
\end{center}
\label{R1R2}
\end{table}

\medskip
The rational functions we have constructed so far lead to the variables $U_0,U_1,U_4,U_7$ of Theorem \ref{main}. The other sections are obtained by
applying the order three automorphism. We used Mathematica to tabulate numerically several dozens points on $\Pp^2_{\fake}$ by first picking random
values for $Y_2$ and $Y_3$, then solving for (one of the) values of $Y_0$, then solving for one of the values of $z$ by taking a seventh root of $z^7$.
Then we looked for degree two and three polynomial equations that vanish on these points. Mathematica is able to work with these numerical approximations by keeping accuracy estimates. As a result, it can give solutions of expected dimension to linear system whose coefficients are only
known approximately by assuming that all minors within the accuracy bound of zero are in fact zero. After finding approximations of the resulting expressions by algebraic numbers, we arrived at $84$ degree three equations
of Theorem \ref{main}

%{\bf Should calculate just for kicks the equation on $U_0,U_1,U_2,U_3$. It could provide a decent looking birational model. Nope, nothing in low degrees.}

\section{Concluding remarks and open questions.}
We have also calculated $147$ degree seven equations among sections of { $4H$} on the unramified double cover of $\Pp^2_{\fake}$. There were no degree six equations.


\begin{thebibliography}{[BPV]}

{ \bibitem[BK]{BK} L. Borisov, J. Keum, \textit{Explicit equations of a fake projective plane}, arXiv:1802.06333v2 [math.AG].}

\bibitem[AK]{AK} D. Allcock, F. Kato, {\em  A fake projective plane via $2$-adic uniformization with torsion}, Tohoku Math. J. 69 (2017) 221-237.

%\bibitem[Arm65]{arm1} M. A. Armstrong, {\em  On the fundamental group of an orbit space, } Proc. Cambridge Philos. Soc. 61 (1965) 639--646.

%\bibitem[Arm68]{arm2} M. A. Armstrong, {\em  The fundamental group of the orbit space of a discontinuous group,} Proc. Cambridge Philos. Soc. 64 (1968) 299--301.



\bibitem[Au]{Aubin} T. Aubin, \textit{\'Equations du type Monge-Amp\`ere sur les vari\'et\'es k\"ahleriennes compactes}, C. R. Acad. Sci. Paris Ser. A-B {\bf 283} (1976), no. 3, Aiii, A119--A121.

\bibitem[BHPV]{bpv}
 W. Barth, K. Hulek, Ch. Peters, A. Van de Ven,
 {\em  Compact complex surfaces. Second edition. } Ergebnisse der Mathematik und ihrer Grenzgebiete. 3. Folge. A Series of Modern Surveys in Mathematics , 4. Springer-Verlag, Berlin, 2004. xii+436 pp.

%\bibitem[Brown82]{brown} K.S. Brown, {\em Cohomology of groups}, Springer GTM 87, 309. pp (1982).

\bibitem[CS]{CS} D. Cartwright, T. Steger,
\textit{Enumeration of the 50 fake projective planes}, C. R. Acad.
Sci. Paris, Ser. I {\bf 348} (2010) 11-13.

\bibitem[CS2]{CS2} D. Cartwright, T. Steger,

http://www.maths.usyd.edu.au/u/donaldc/fakeprojectiveplanes

%\bibitem[Cat15]{topmethods} \textsc{F. Catanese:} Topological methods in moduli theory. Bull. Math. Sci. 5, No. 3,  287--449 (2015).

% \bibitem[CF96]{cf}  Catanese, Fabrizio; Franciosi, Marco;  {\em  Divisors of small genus on algebraic surfaces and projective embeddings.}  Proceedings of the Hirzebruch 65 Conference on Algebraic Geometry (Ramat Gan, 1993), 109--140, Israel Math. Conf. Proc., 9, Bar-Ilan Univ., Ramat Gan, 1996.

%\bibitem[CFHR99]{4names} Catanese, Fabrizio; Franciosi, Marco; Hulek, Klaus; Reid, Miles; {\em Embeddings of curves and surfaces. } {\bf Nagoya Math. J. 154 } (1999) 185--220.

\bibitem[CK]{CK}
F. Catanese, J. Keum,
{\em The Bicanonical map of fake projective planes with an automorphism.}
arXix:1801.05291, International Mathematics Research Notices, published online.

\bibitem[CD]{CD} F. Cossec,  I. Dolgachev,  \textit {Enriques surfaces I}, Birkh\"auser 1989

\bibitem[DBDC]{dbdc} G. Di Brino, L. Di Cerbo, \textit{Exceptional collections and the bicanonical map of Keum's fake projective planes}, Communications in Contemporary Mathematics {\bf 20} no. 1 (2018), 1650066 (13 pages),
World Scientific Publishing Company,
DOI: 10.1142/S0219199716500668.


\bibitem[D]{D} I. Dolgachev, \textit{Algebraic surfaces with $q=p_g=0$}, C.I.M.E. Algebraic surfaces, pp 97-215, Liguori Editori, Napoli
1981.

%\bibitem[F]{F} N. Fakhruddin, \textit{Exceptional collections on 2-adically uniformised fake projective planes}, Math. Res. Lett. {\bf 22} (2015) 43-57.



%\bibitem[HK1]{HK1} D. Hwang and J. Keum, \textit{The maximum number of singular points on rational homology projective planes}, J. Algebraic Geom. {\bf 20} (2011) 495-523.

%\bibitem[HK2]{HK2} D. Hwang and J. Keum, \textit{Algebraic Montgomery-Yang Problem: the nonrational surface case}, Michigan Math. J. {\bf 62} (2013) 3-37.


\bibitem[GKMS]{GKMS} S. Galkin, L. Katzarkov, A. Mellit, E. Shinder, \textit{Derived categories of Keum's fake projective planes}
Adv. Math. {\bf 278} (2015) 238-253.

%\bibitem[Jacob-2-80]{jac2} \textsc{N. Jacobson:} Basic Algebra II, W.H. Freeman and co., San Francisco (1980), 666 pp.


\bibitem[I]{Ish} M-N Ishida, \textit{An elliptic surface covered by Mumford's fake projective plane}, Tohoku Math.J. (2) {\bf 40} no. 3 (1988) 367-396.

\bibitem[IK]{IK} M-N Ishida, F Kato, \textit{The strong rigidity theorem for non-Archimedean uniformization},
Tohoku Math. J. (2) {\bf 50} (1998) 537--555.

\bibitem[K06]{K06} J. Keum, \textit{A fake projective plane with an order 7 automorphism}, Topology {\bf 45} (2006) 919-927.

\bibitem[K08]{K08} J. Keum, \textit{Quotients of fake projective planes}, Geom. Topol. {\bf 12} (2008) 2497-2515.

\bibitem[K11]{K11} J. Keum, \textit{A fake projective plane constructed from an elliptic surface with multiplicities $(2,4)$}, Sci. China Math. {\bf 54} (2011) 1665-1678.

\bibitem[K12]{K12} J. Keum, \textit{Toward a geometric construction of fake projective planes}, Rend. Lincei Mat. Appl. {\bf 23} (2012) 137-155.
\bibitem[K13]{K13}J. Keum, \textit{Every fake projective plane with an order 7 automorphism has $H^0(2L)=0$ for any ample generator $L$},
manuscript circulated on July 8 2013.

\bibitem[K17]{K17} J. Keum, \textit{Vanishing theorem on fake projective planes with enough automorphisms}, Trans. Amer. Math. Soc. {\bf 369}     (2017) 7067-7083.
\bibitem[KK]{kk} V. S.  Kharlamov, V. M. Kulikov, \textit{On real structures on rigid surfaces}, Izv. Russ. Akad. Nauk. Ser. Mat. {\bf 66}, no. 1, (2002)  133-152; Izv. Math. {\bf 66}, no. 1, (2002) 133-150.


\bibitem[Kl]{Kl} B. Klingler, \textit{Sur la rigidit\'e de certains groupes fondamentaux, l'arithm\'eticit\'e des r\'eseaux hyperboliques complexes, et les "faux plans projectifs"}, Invent. Math. {\bf 153} (2003) 105-143.

%\bibitem[Ko]{Ko} J. Koll\'ar, \textit{Shafarevich Maps and Automorphic Forms}, Princeton University Press, Princeton 1995.

\bibitem[Mos]{Mos} G. D. Mostow, \textit{Strong
rigidity of locally symmetric spaces}, Annals Math. Studies {\bf
78}, Princeton Univ. Press, Princeton, N.J.; Univ. Tokyo Press,
Tokyo 1973.

\bibitem[M]{Mum} D. Mumford, \textit{ An algebraic surface with K ample, $K^2=9$, $p_g=q=0$}, Amer. J. Math. {\bf 101} (1979) 233-244.


\bibitem[PY]{PY} G. Prasad, S.-K. Yeung, \textit{Fake projective planes}, Invent. Math. {\bf 168} (2007)
321-370; Addendum, {\bf 182 } (2010) 213-227.

\bibitem[R]{reider}  I. Reider,  \textit{ Vector bundles of rank 2 and linear systems on algebraic surfaces.}
 Ann. of Math. (2) {\bf 127} (1988), no. 2, 309--316

\bibitem[Y]{Yau} S.-T. Yau, \textit{Calabi's conjecture and some new results in algebraic geometry},
Proc. Natl. Acad. Sci. USA {\bf 74} (1977) 1798--1799.

%\bibitem[Yau78]{Yau78} S.-T. Yau, \textit{A general Schwarz lemma for K\"ahler manifolds} Amer. J. Math. {\bf 100} (1978) 197--203.
\end{thebibliography}
\end{document}